\documentclass[12pt,a4paper]{amsart}
\usepackage[latin1]{inputenc}
\usepackage{latexsym}
\usepackage{graphicx,shortvrb}
\usepackage[usenames,dvipsnames]{color}
\usepackage{amsmath, amssymb}
\usepackage{amsfonts}
\usepackage[colorlinks, bookmarks=true]{hyperref}

\newtheorem{theorem}{Theorem}[section]
\newtheorem{lemma}[theorem]{Lemma}
\newtheorem{proposition}[theorem]{Proposition}
\newtheorem{corollary}[theorem]{Corollary}

\theoremstyle{definition}
\newtheorem{definition}[theorem]{Definition}
\newtheorem{remark}[theorem]{Remark}

\newcommand{\field}[1] {\mathbb{#1}}
\newcommand{\R}{\field{R}}
\newcommand{\dens}{{\mathbf{d}}}
\newcommand{\N}{\mathbb{N}}
\newcommand{\C}{\mathbb{C}}
\newcommand{\Z}{\mathbb{Z}}
\newcommand{\D}{\mathbb{D}}
\newcommand{\T}{\mathbb{T}}
\newcommand{\one}{\mathbf{1}}
\newcommand{\spn}{\mathbf{span}}
\newcommand{\rank}{\mathbf{rank}}

\definecolor{orange}{rgb}{1,0.5,0}

\title{Approximation schemes satisfying Shapiro's Theorem}
\author{J. M. Almira and T. Oikhberg}

\begin{document}

\keywords{Approximation scheme, approximation error, Bernstein's Lethargy,
approximation by dictionary}
\subjclass[2000]{41A25, 41A65, 41A27}

\baselineskip=16pt

\numberwithin{equation}{section}

\maketitle \markboth{Approximation schemes satisfying Shapiro's Theorem}{J. M. Almira, T. Oikhberg}
\begin{abstract}
An approximation scheme is a family of homogeneous subsets $(A_n)$ of a
quasi-Banach space $X$, such that
$A_1 \subsetneq A_2 \subsetneq \ldots \subsetneq X$, $A_n + A_n \subset A_{K(n)}$,
and $\overline{\cup_n A_n} = X$. Continuing the line of research originating
at the classical paper \cite{bernsteininverso} by S.N.~Bernstein, we give several characterizations of the approximation schemes with the property that, for every sequence $\{\varepsilon_n\}\searrow 0$, there exists $x\in X$ such that
$dist(x,A_n)\neq \mathbf{O}(\varepsilon_n)$ (in this case we say that $(X,\{A_n\})$
satisfies Shapiro's Theorem). If $X$ is a Banach space, $x \in X$ as above exists
if and only if, for every sequence $\{\delta_n\} \searrow 0$, there exists $y \in X$
such that $dist(y,A_n) \geq \delta_n$. We give numerous examples of approximation schemes
satisfying Shapiro's Theorem.
\end{abstract}

\section{Introduction and motivation}\label{intro}
One of the most remarkable early results in the constructive theory of functions
is Bernstein Lethargy Theorem: if
$X_0\subsetneq X_1\subsetneq X_2\subsetneq \cdots \subsetneq X$ is an ascending chain of finite dimensional vector subspaces of a Banach space $X$, and $\{\varepsilon_n\}\searrow 0$
is a non-increasing sequence of positive real numbers that converges to zero, then there exists an element $x\in X$ such that the $n$-th error of best approximation by elements
of $X_n$ satisfies $E(x,X_n)=\varepsilon_n$ for all $n\in\mathbb{N}$.
Here and throughout the paper, we write $E(x,A) = \inf_{a \in A} \|x-a\|$
($x$ and $A$ are an element and a subset of a quasi-Banach space $X$, respectively).
Furthermore, the notation $\{\alpha_n\} \searrow 0$ means that the sequence
$(\alpha_n)$ is non-increasing, and $\lim \alpha_n = 0$. 

The result quoted above was first obtained in 1938 by S.N. Bernstein \cite{bernsteininverso}
for $X=C([0,1])$ and $X_n=\Pi_n$, the vector space of real polynomials of degree $\leq n$.
The case of arbitrary finite dimensional $X_n$ is treated, for instance, in
\cite[Section II.5.3]{singerlibro}.

There are very few generalizations of Bernstein's result to arbitrary chains of
(possibly infinite dimensional) closed subspaces
$X_1 \subsetneq X_2 \subsetneq \ldots$ of a Banach space $X$.
The results due to Tjuriemskih \cite{tjuriemskih} and Nikolskii \cite{nikolskii,nikolskii2}
(see also \cite[Section I.6.3]{singerlibro}) assert that a sufficient (resp.~necessary)
condition for the existence of $x \in X$ verifying $E(x,X_n) = \varepsilon_n$ is that
$X$ is a Hilbert space (resp.~$X$ is reflexive). These results were proved independently and
by other means by Almira and Luther \cite{almiraluther1, almiraluther2} and Almira and Del Toro \cite{almiradeltoro1}. Moreover, in \cite{almiradeltoro2} it was shown that if $X$ is a reflexive Banach space and $\{0\}\subset X_1\subset X_2 \subset \cdots$ is an infinite chain of closed subspaces of $X$ then for every pair of sequences of positive numbers $\{\varepsilon_n\}\searrow 0$, $\{\delta_n\}\searrow 0$, there is an element $x\in X$ such that $E(x,X_n)/\varepsilon_n$ converges to zero but $E(x,X_n)/\varepsilon_n\not= \mathbf{O}(\delta_n)$.
Also, Bernstein
Lethargy Theorem has been generalized to
chains of finite-dimensional subspaces in non-Banach spaces (such as $SF$-spaces)
by G. Lewicki \cite{lewicki, lewicki1}. These two approaches were successfully combined
by Micherda \cite{micherda}.

Thanks to the work by Plesniak \cite{plesniak}, the lethargy theorem has become a very
useful tool for the theory of quasianalytic functions of several complex variables.

In 1964 H.S. Shapiro \cite{shapiro} used Baire Category Theorem to prove that,
for any sequence $X_1 \subsetneq X_2 \subsetneq \ldots \subsetneq X$ of
closed (not necessarily finite dimensional) subspaces of a Banach space $X$,
and any sequence $\{\varepsilon_{n}\}\searrow 0$, there exists an $x\in X$ such that $E(x,X_{n})\neq\mathbf{O}(\varepsilon_{n})$.  This result was strengthened by Tjuriemskih \cite{tjuriemskih1} who, under the very same conditions of Shapiro's Theorem, proved the existence of  $x\in X$ such that $E(x,X_{n})\geq \varepsilon_{n}$, $n=0,1,2,\cdots$. Moreover, Borodin \cite{borodin} gave an easy proof of this result and proved that, for arbitrary infinite dimensional Banach spaces $X$ and for sequences $\{\varepsilon_n\}\searrow 0$ satisfying $\varepsilon_n>\sum_{k=n+1}^\infty\varepsilon_k$, $n=0,1,2,\cdots$, there exists  $x\in X$ such that $E(x,X_{n})= \varepsilon_{n}$, $n=0,1,2,\cdots$.

However, approximation by linear subspaces of a Banach space is very restrictive.
There are many other choices of approximation processes such as rational approximation, approximation by splines with of without free knots, $n$-term approximation with
dictionaries of different kinds, and approximation of operators by operators of
finite rank, just to mention a few of them. Do the results of Bernstein, Shapiro and Tjuriemskih
hold in this setting, too? The following startling result was proved by
Yu.~Brudnyi \cite[Theorem 4.5.12]{brundykrugljak}:

\begin{theorem}\label{brud_original}
Suppose $\{0\}=A_0\subset A_1\subset\cdots \subset A_n\subset$ is an infinite chain
of subsets of a Banach space $X$, satisfying the following conditions:
$A_n+A_m\subset A_{n+m}$ for all $n,m\in\mathbb{N}$;
$\lambda A_n\subset A_n$ for all $n\in\mathbb{N}$ and all scalars $\lambda$;
$\bigcup_{n\in\mathbb{N}}A_n$ is dense in $X$; and
\begin{eqnarray}\label{condicionbrudnyi}
\gamma=\inf_{n\in\mathbb{N}} \sup_{x \in A_{n+1}, \|x\| \leq 1}E(x,A_n)>0 .
\end{eqnarray}
Then for every non-increasing convex sequence
$\{\varepsilon_n\}_{n=0}^{\infty}\searrow 0$ there exists $x\in X$ such that
$E(x,A_n)\geq \varepsilon_n$ for all $n\in\mathbb{N}$,
and $E(x,A_n) \leq c \varepsilon_n$ for infinitely many values of $n$
(the constant $c$ depends only on $\gamma$).
\end{theorem}

Recall that a sequence ${\varepsilon_n}$ is called {\it convex} if, for every $n$,
$\varepsilon_n \leq (\varepsilon_{n-1} + \varepsilon_{n+1})/2$.
By \cite[pp.113-114]{edwards}, for any sequence $\{\varepsilon_n\}\searrow 0$,
there is a convex sequence $\{\xi_n\}\searrow 0$ such that $\xi_n\geq \varepsilon_n$ for all
$n\in \mathbb{N}$. Thus, we do not need to assume the convexity of $\{\varepsilon_n\}$
to show the existence of $x \in X$ satisfying
$E(x,A_n)\geq \varepsilon_n$ for any $n\in\mathbb{N}$.

In this paper, we are concerned with generalizations of results of Brudnyi and
Shapiro quoted above for general approximation schemes, defined by
A.~Pietsch \cite{Pie} to produce a unified approach to diverse phenomena of
approximation theory.

\begin{definition}\label{def_appr_scheme}
Suppose $X$ is a quasi-Banach space, and let
$A_0\subset A_1 \subset \ldots \subset X$
be an infinite chain of subsets of $X$, where all inclusions are
strict. We say that $(X,\{A_n\})$ is an {\it approximation scheme}
(or that $(A_n)$ is an approximation scheme in $X$) if:
\begin{itemize}
\item[$(i)$] There exists a map $K:\mathbb{N}\to\mathbb{N}$ such that $K(n)\geq n$ and $A_n+A_n\subseteq A_{K(n)}$ for all $n\in\mathbb{N}$
(we can assume that $K$ is increasing).

\item[$(ii)$] $\lambda A_n\subset A_n$ for all $n\in\mathbb{N}$ and all scalars $\lambda$.

\item[$(iii)$] $\bigcup_{n\in\mathbb{N}}A_n$ is dense in $X$.
\end{itemize}
\end{definition}

One example of an approximation scheme is an increasing chain of linear subspaces
of $X$, whose union is dense. Then we can take $K(n) = n$. Further examples of
approximation schemes can be found throughout the paper.

\begin{definition}\label{def}
We say that  $(X,\{A_n\})$ {\it satisfies Shapiro's Theorem} if for any non-increasing sequence $\{\varepsilon_n\}\searrow 0$ there exists some $x\in X$ such
that $E(x,A_{n})\neq\mathbf{O}(\varepsilon_{n}).$
\end{definition}


Section \ref{shapiro} is devoted to describing approximation schemes satisfying
Shapiro's Theorem (Theorems \ref{maintheorem} and \ref{condicion}).
In Section \ref{section_brudnyi}, we prove that for an approximation scheme in a Banach space $X$, satisfying Shapiro's Theorem is equivalent to (a weakened version of) Brudnyi's Theorem
\ref{brud_original} (Theorem~\ref{shapiroimpliesbrudnyi}, Corollary~\ref{coro}).
Section \ref{fail_shapiro} shows some examples of ``pathological'' approximation schemes
failing Shapiro's Theorem. Section \ref{seccion_central} studies the relationship between approximation
schemes that satisfy Shapiro's Theorem, and those verifying the abstract versions of Jackson's and Bernstein's inequalities.   Section~\ref{examples}
contains many examples of
approximation schemes which do satisfy Shapiro's Theorem. Finally, Section \ref{fast_decay} examines the
related question of controlling the rate of decay of the best approximation  errors.


\section{Shapiro's Theorem}\label{shapiro}

Throughout this paper, we work with approximation schemes in infinite
dimensional quasi-Banach spaces. The proposition below shows that a finite
dimensional space cannot ``host'' an approximation scheme.

\begin{proposition}\label{finite_dimension}
Suppose $X$ is a finite dimensional space, and the family of its subsets
$A_0\subset A_1\subset\cdots \subset A_n\subset\cdots \subset X$
satisfies (i), (ii), and (iii) of Definition~\ref{def_appr_scheme}.
Then there exists $N\in \N$ such that $A_N=X$.
\end{proposition}

\begin{proof} For each $n$, $X_n = \spn[A_n]$ is a closed subspace of $X$.
Then $X_1 \subset X_2 \subset \cdots$. As $\cup_n A_n$ is dense
in $X$, we conclude that $X_n = X$ for some $n$. By Caratheodory's Theorem, and by the homogeneity of the set $A_n$,
any $x \in X$ can be represented as $x = \sum_{k=1}^M \alpha_k a_k$, with
$a_k \in A_n$, and $\alpha_k \in \R$ (here $M = \dim X + 1$). Therefore,
$X = A_N$, where $N = K(\ldots(K(n))\ldots)$ ($M$ times).
\end{proof}


Note that if $((X,\|\cdot\|),\{A_n\})$ satisfies Shapiro's Theorem, and $|||\cdot|||$
is an equivalent quasi-norm on $X$, then
$((X,|||\cdot|||),\{A_n\})$ also satisfies Shapiro's Theorem.
This remark will be particularly useful for quasi-normed spaces $X$, as it allows us
to deploy Aoki-Rolewicz theorem: any quasi-normed space can be equipped with an equivalent
norm $||| \cdot |||$ for which there exists $p \in (0,1]$ such that
$|||x+y|||^p\leq |||x|||^p + |||y|||^p $ for any $x, y \in X$
(see e.g.~\cite[pp. 7-8]{kalton}).

\begin{theorem}\label{maintheorem}
Suppose $(A_n)$ is an approximation scheme in a quasi-Banach space $X$.
The following are equivalent:
\begin{itemize}
\item[$(a)$]  The approximation scheme $(X,\{A_n\})$ satisfies Shapiro's Theorem.

\item[$(b)$] There exists a constant $c>0$ and an infinite set $\mathbb{N}_{0}\subseteq\mathbb{N}$ such that for all $n\in\mathbb{N}_{0}$, there
exists $x_{n}\in X\setminus \overline{A_{n}}$ which satisfies $E(x_{n},A_{n})\leq cE(x_{n},A_{K(n)}).$

\item[$(c)$]
There is no sequence $\{\varepsilon_n\}\searrow 0$ such that
$E(x,A_n)\leq \varepsilon_n\|x\|$  for all $x\in X$ and $n\in \N$.
\end{itemize}
\end{theorem}

For the proof we need:




\begin{lemma}\label{sucesiones} Let
$h:\mathbb{N}\to\mathbb{N}$ be a map such that $h(n)\geq n$ for all $n$,
 and let $\{\varepsilon _{n}\}\searrow 0$. Then there exists a
sequence
$\{\xi_{n}\}\searrow 0$ such that $\xi_{n}\geq\varepsilon_{n}$ and
$\xi_{n}\leq 2\xi_{h(n)}$ for every $n$.
\end{lemma}

\begin{proof} 
Passing from the original function $h$ to (say)
$h^\prime(n) = \max_{1 \leq k \leq n} h(k) + n$, we can assume that
(i) $h(n) > n$ for every $n$, and (ii) the function $h$ is strictly increasing.
Set $m_0 = 0$, and, for $k \geq 1$, $m_k = h(m_{k-1})$. Set $\beta_0 = \varepsilon_1$,
and $\beta_k = \max\{\varepsilon_{m_k}, \beta_{k-1}/2\}$ for $k \geq 1$.
For $n \in \N$, find $k \geq 0$ such that  $n \in [m_k, m_{k+1})$, and set
$\xi_n = \beta_k$.

Then the sequence $(\xi_n)$ has the desired properties.
For $n \in [m_k, m_{k+1})$, $\xi_n = \beta_k \geq \varepsilon_{m_k} \geq \varepsilon_n$.
Furthermore, as $h$ is increasing, $h(n) \in [m_{k+1}, m_{k+2})$, hence
$\xi_{h(n)} = \beta_{k+1} \geq \beta_k/2 = \xi_n/2$.
It remains to show that $\lim \xi_n = 0$, or in other words, that
$\lim \beta_k = 0$. If $\beta_k = \varepsilon_{m_k}$ for infinitely
many values of $k$, then $\lim \beta_k = \lim \varepsilon_{m_k} = 0$.
Otherwise, $\beta_k = \beta_{k-1}/2$ for any $k \geq k_0$. In this case,
too, $\lim \beta_k = 0$.
\end{proof} 

\begin{proof}[Proof of Theorem~\ref{maintheorem}] 
As $X$ is a quasi-Banach space, there exists a constant $C_X$ such that
$\|x+y\| \leq C_X(\|x\| + \|y\|)$ for any $x,y \in X$.

$(b)$ $\Rightarrow$ $(a)$:
As a first step, we prove the existence of $x \in X$ satisfying
$E(x,A_n) \neq \mathbf{O}(\varepsilon_n)$ under the additional assumption that
$\varepsilon_n\leq 2\varepsilon_{K(n+1)-1}$ for all $n\in\mathbb{N}$. Assume, for
the sake of contradiction, that $E(x,A_{n})=\mathbf{O}(\varepsilon_{n})$ for all $x\in X$.
Then $X=\bigcup_{m=1}^{\infty}\Gamma_m$, where $\Gamma_{\alpha}=\{x\in X:E(x,A_n)\leq \alpha \varepsilon_n, n=0,1,2,\cdots\}$ ($\alpha>0$).
The sets $\Gamma_m$ are closed subsets of $X$. Furthermore,
$E(-x,A_n)=E(x,A_n)$ for all $n$, hence $\Gamma_m=-\Gamma_m$ for all $m$.
Finally,
\begin{equation}
{\mathbf{conv}}(\Gamma_m) \subset \Gamma_{2mC_X}
\label{eq:contained}
\end{equation}
(here, ${\mathbf{conv}}(S)$ stands for the convex hull of a set $S$).
Indeed, suppose $x, y \in \Gamma_m$, and $\lambda \in [0,1]$. Recalling the inclusion
$A_n + A_n \subset A_{K(n)}$, we see that, for every $n$,
\begin{eqnarray*}
E(\lambda x + (1-\lambda)y ,A_{K(n)}) &=& \inf_{g\in A_{K(n)}}\|\lambda x+ (1-\lambda)y-g\|\\
&\leq&  \inf_{a,b\in A_{n}}\|\lambda (x-a) +(1-\lambda)(y-b)\| \\
&\leq& C_X[ \inf_{a\in A_{n}}\|\lambda (x-a)\| +\inf_{b\in A_{n}}\|(1-\lambda)(y-b)\|] \\
&=& \lambda C_XE(x,A_n)+ (1-\lambda)C_XE(y,A_n)\leq m C_X\varepsilon_n .
\end{eqnarray*}
For an arbitrary $j$, find $n$ such that $K(n) \leq j < K(n+1)$
(for simplicity, we set $K(0) = 0$). Then
$$
E(\lambda x + (1-\lambda)y ,A_j) \leq E(\lambda x + (1-\lambda)y ,A_{K(n)}) \leq
m C_X\varepsilon_n \leq 2 m C_X \varepsilon_{K(n+1)-1} \leq 2 m C_X \varepsilon_j ,
$$
which implies $\lambda x + (1-\lambda)y \in \Gamma_{2mC_X}$, thus proving
\eqref{eq:contained}.

By Baire Category Theorem, there exists some
$m_0\in\mathbb{N}$ such that $\Gamma_{m_0}$ has non-empty interior. That is, there exists
a ball $B(x,r)\subset \Gamma_{m_0}$ with $r>0$. By symmetry, $-B(x,r)\subset \Gamma_{m_0}$.
By \eqref{eq:contained},
$$
B(0,r) \subset \frac{1}{2} \big( B(-x,r) + B(x,r) \big)
\subset \Gamma_{2m_0C_X} .
$$
Hence,
$\frac{r}{\|x\|}x\in \Gamma_{2m_0C_X}$ for every $x\in X$, and the inequality
\[
E(x,A_n)\leq \frac{\|x\|}{r}2m_0C_X\varepsilon_n
\]
holds for all $x\in X$ and all $n\in\mathbb{N}$.

For $n\in\mathbb{N}_0$, find $a_n\in A_n$ such that  $\|x_n-a_n\|\leq 2E(x_n,A_n)$, where
$\{x_k\}_{k\in\mathbb{N}_0}$ is the sequence of elements of $X$  given by condition $(b)$.
Take $y_n=x_n-a_n$. Then
\begin{equation*}
\|y_{n}-b_{n}\|=\|x_{n}-(a_{n}+b_{n})\|\geq E(x_{n},A_{K(n)})\geq\frac{1}{c}E(x_{n},A_{n})\geq\frac{1}{2c}\|y_{n}\|
\end{equation*}
for all $b_{n}\in A_{n}$. Hence
\[
\frac{1}{2c}\|y_{n}\|\leq E(y_{n},A_{n}) \leq   \frac{\|y_n\|}{r}2m_0C_X\varepsilon_n ,
\]
and consequently, $1/(2c) \leq 2m_0C_X\varepsilon_n/r$ for all $n\in\mathbb{N}_0$.
This contradicts $\varepsilon_n\to 0$. Thus, for every sequence
$\{\varepsilon_n\} \searrow$ satisfying $\varepsilon_n\leq 2\varepsilon_{K(n+1)-1}$
($n \in \N$), there exists $x \in X$ such that $E(x,A_n) \neq \mathbf{O}(\varepsilon_n)$.

Now suppose the sequence $\{\varepsilon_n\} \searrow 0$ is arbitrary.
Applying Lemma \ref{sucesiones} to $\{\varepsilon_n\}_{n=0}^{\infty}$
and the map $h(n)=K(n+1)-1$, we obtain a sequence $\{\xi_n\}_{n=0}^{\infty}$
satisfying $\varepsilon_n \leq \xi_n\leq 2\xi_{K(n+1)-1}$
for all $n\in\mathbb{N}$. By the above, there exists $x\in X$ such that
$E(x,A_n)\not= \mathbf{O}(\xi_n)$, which implies $E(x,A_n)\not= \mathbf{O}(\varepsilon_n)$.
This ends the proof of $(b)\Rightarrow (a)$.

$(a)$ $\Rightarrow$ $(b)$:
If $X=\cup_{n=0}^{\infty}\overline{A_{n}}$, then both $(a)$ and $(b)$ are false, since in
this case, for any $x \in X$ there exists $n \in \N$ such that $E(x,A_n) = 0$.
Suppose $X \neq \cup_{n=0}^{\infty}\overline{A_{n}}$, and (b) is false. Then
the sequence $\{c_{n}\}_{n=0}^{\infty}\subset\lbrack 0,\infty)$, given by
\[
c_{n}=\inf_{x\in X\setminus\overline{A_{K(n)}}}\frac{E(x,A_{n})}{E(x,A_{K(n)})} ,
\]
has no bounded subsequences, hence $\lim_{n\rightarrow\infty}c_{n}=\infty$. Set $\varepsilon_{k}=1/c_{n}$ for $K(n) \leq k < K(n+1)$ and let $\{\varepsilon_n^*\}$ denote the non-increasing rearrangement of the sequence $\{\varepsilon_n\}\in c_0(\mathbb{N})$. For any
$x\in X\setminus\cup_{n=0}^{\infty}\overline{A_{n}}$, and any
$k\in\lbrack K(n),K(n+1))$, we have
\begin{equation}\label{shapiro_fails}
E(x,A_{k})\leq E(x,A_{K(n)})\leq\frac{1}{c_{n}}E(x,A_{n})\leq\frac{1}{c_{n}}\|x\|=\varepsilon_{k}\|x\| \leq \varepsilon_{k}^*\|x\| ,
\end{equation}
hence $E(x,A_{k})=\mathbf{O}(\varepsilon_{k}^*)$, and $(a)$ is also false.

$(a)$ $\Rightarrow$ $(c)$ is clear. 
On the other hand, if $(a)$ is false then $(b)$ is also false, so that \eqref{shapiro_fails} holds true. This implies that
$E(x,A_{k})\leq \varepsilon_{k}^*\|x\|$, for the sequence $\{\varepsilon_k^*\} \searrow 0$
described above.
\end{proof} 

\begin{remark}
It follows from Theorem \ref{maintheorem} that every  non trivial
linear approximation scheme (i.e. every approxi\-mation scheme verifying $K(n)=n$
and $\overline{A_n}\neq X$ for all $n$) satisfies Shapiro's Theorem.
In particular, this extends Shapiro's result to the quasi-Banach setting.

A different proof of Theorem \ref{maintheorem} was given by Almira and Del Toro in
\cite{almiradeltoro1, almiradeltoro2}. That proof used some
general theory of approximation spaces, introduced by Almira and Luther in
\cite{almiraluther1, almiraluther2}. The proof presented here is self-contained,
avoids the theory of generalized approximation spaces,
and follows a more classical line of thinking.
\end{remark}

One of our main tools for verifying that an approximation scheme satisfies
Shapiro's Theorem is property (P).

\begin{definition}\label{def:prop_P}
We say that an approximation scheme $(X,\{A_n\})$ satisfies property $(P)$ (with constants $a,b>0$) if for every $n\in\mathbb{N}$, $n>0$, there exists an element $x\in X$ with $\|x\|=1$ such that $E(x,A_n)\geq \frac{1}{an^b}$.
\end{definition}

\begin{theorem}\label{condicion}
Suppose an approximation scheme $(X,\{A_n\})$ satisfies property $(P)$,
and there exists $c > 1$ such that $A_n+A_n\subseteq A_{c n}$ for any $n\in\mathbb{N}$.
Then  $(X,\{A_n\})$ satisfies Shapiro's Theorem.
\end{theorem}

\begin{proof}
Assume, for the sake of contradiction, that $(X,\{A_n\})$ fails Shapiro's Theorem.
By Theorem \ref{maintheorem}, for any $C>1$ there exists $N\in\mathbb{N}$ such that
$E(x,A_n)\geq CE(x,A_{c n})$ for any $x\in X$ and $n\geq N$. Pick $C>c^b$ and select $k$ to satisfy $aN^b < \frac{C^k}{c^{bk}}$
(here, $a$ and $b$ are as in Definition~\ref{def:prop_P}).
Take $x\in X$ with $\|x\|=1$ and $E(x,A_{c^kN})\geq \frac{1}{a(c^kN)^b}$. Then
\[
1=\|x\|\geq E(x,A_N)\geq CE(x,A_{c N})\geq C^2E(x,A_{c^2 N})\geq \cdots\geq C^kE(x,A_{c^k N}),
\]
so that
\[
\frac{1}{a(c^kN)^b}\leq E(x,A_{c^k N})\leq C^{-k} ,
\]
hence $a(c^kN)^b\geq C^k$, which contradicts our choice of $k$.
\end{proof} 

Section~\ref{examples} contains several examples where Property (P)
is used to show that an approximation scheme satisfies Shapiro's Theorem.

\section{A comparison with Brudnyi's theorem}\label{section_brudnyi}

To proceed, we need to introduce some notation.
Recall that, for $x \in X$ and $A \subset X$, we define
$E(x,A) = \inf_{a \in A}\|x-a\|$. Furthermore, for subsets $A, B$ of $X$, we define
$E(B,A)=\sup_{b\in B}E(b,A)$
(note that $E(A,B)$ may be different from $E(B,A)$).
We denote by $S(X)$ the unit sphere of a quasi-Banach space $X$.

\begin{definition}\label{weak_brudnyi}
We say that $(X,\{A_n\})$ {\it satisfies Brudnyi's condition} if
\eqref{condicionbrudnyi} holds. We say that $(X,\{A_n\})$
{\it satisfies weak Brudnyi's condition with constant $c\in (0,1]$} if
$E(S(X),A_n)\geq c$ for all $n\in\mathbb{N}$.
\end{definition}

Note that Brudnyi's condition implies the ``jump condition''
from Theorem~\ref{maintheorem}(b), that is, the existence
(for each $n \in \N$) of $x_n \in X$ satisfying $E(x_n,A_n)\leq C E(x_n,A_{K(n)})$.
This implication holds for general approximation schemes,
and not just for the case $K(n)=2n$, covered by Brudnyi's theorem.
Indeed, applying \eqref{condicionbrudnyi} to $A_{K(n)}$, we obtain
$x_{n}\in A_{K(n)+1}$ such that  $\|x_{n}\|=1$ and $E(x_{n},A_{K(n)})\geq \gamma$.
Then
\[
E(x_{n},A_n)\leq 1=C\gamma\leq C E(x_{n},A_{K(n)}),
\]
where $C = 1/\gamma$.

However, there exist approximation schemes failing Brudnyi's condition
\eqref{condicionbrudnyi}, for which one can obtain a prescribed rate of
decay of $(E(x,A_n))$.

\begin{theorem} There exists an approximation scheme $(A_n)$ in the space $c_0$,
such that $A_m + A_n \subset A_{\max\{m,n\}+1}$ for any $m,n \in \N$, and:
\begin{enumerate}
\item
Brudnyi's condition \eqref{condicionbrudnyi} is not satisfied.
\item
For any $\{\varepsilon_n\}\searrow 0$, there exists $x \in c_0$ such that
$E(x, A_{2n-1}) = \varepsilon_n$ for any $n \in \N$.
Consequently, $(A_n)$ satisfies Shapiro's Theorem.
\end{enumerate}
\end{theorem}

\begin{proof} 
We introduce the sets $B_n$:
$B_0=\{\mathbf{0}\}$, $B_1=\{(x_1,0,\cdots,0,\cdots):x_1\in\mathbb{R}\}$ and, for $n\geq 1$,
\[
B_{n+1}=\Big\{(x_1,\cdots,x_{n+1}, 0,\cdots):(x_1,\cdots,x_n)\in\mathbb{R}^n \text{ and }
|x_{n+1}|\leq\frac{\sup_{k\leq n}|x_k|}{n+1}\Big\}.
\]
Let us also introduce the sets $\Pi_n=\{(x_1,\cdots,x_n,0,\cdots):(x_1,\cdots,x_n)\in\mathbb{R}^n\}$.
Consider the approximation scheme $(X,\{A_n\}_{n=0}^{\infty})$, where
$A_0=B_0$, $A_1=B_1=\Pi_1$, $A_2=B_2$, $A_3=\Pi_2$, $A_4=B_3$, $A_5=\Pi_3$, $\cdots$.
Clearly, $A_0 \subsetneq A_1\subset A_2 \subsetneq \cdots \subsetneq c_0$,
$A_n+A_m\subset A_{\max\{n,m\}+1}\subset A_{n+m}$ for any $m$ and $n$, and
$\overline{\cup_n A_n} = c_0$. Furthermore, if $\{\varepsilon_n\} \searrow 0$,
then $x = (\varepsilon_0, \varepsilon_1, \varepsilon_2, \ldots) \in c_0$
satisfies $E(x,A_{2n-1}) = \varepsilon_n$ for any $n$.

However, there is no $\gamma > 0$ such that  $E(S(X) \cap A_{n+1}, A_n) \geq \gamma$
for every $n$. Indeed, it is easy to see that
$$
E(S(X) \cap A_{2k}, A_{2k-1}) = E(S(X) \cap B_{k+1}, \Pi_k) = \frac{1}{k+1} .
$$
Thus, the approximation scheme $(A_n)$ has the desired properties.
\end{proof} 

Clearly, if, for an approximation scheme $(X, \{A_n\})$, the lower estimate of
Brudnyi's Theorem~\ref{brud_original} holds (that is, for any
$\{\varepsilon_n\} \searrow 0$ there exists $x \in X$ such that $E(x,A_n) \geq \varepsilon_n$ for every $n$), then $(X, \{A_n\})$ satisfies
Shapiro's Theorem. If $X$ is a Banach space, the converse is also true.

\begin{theorem} \label{shapiroimpliesbrudnyi}
Suppose $X$ is a Banach space, and an approximation scheme $(X,\{A_n\})$ satisfies Shapiro's Theorem. 
Then for every sequence $\{\varepsilon_n\}_{n=0}^{\infty}\searrow 0$ there exists
$x\in X$ such that $E(x,A_n)\geq \varepsilon_n$ for all $n\in\mathbb{N}$.
\end{theorem}



For the proof we need two lemmas.
The first one will be stated for the quasi-Banach setting because we will use it later (see Corollary \ref{coro}) to give a new characterization of approximation schemes that satisfy Shapiro's Theorem.

\begin{lemma}\label{brud1}
If $X$ is a quasi-Banach space, and an approximation scheme $(X,\{A_n\})$
satisfies Shapiro's Theorem, then $E(S(X),A_n)=1$ for $n=0,1,2,\ldots$.
\end{lemma}

\begin{proof}
Suppose otherwise. Then there exists $n\in\mathbb{N}$
such that $E(S(X),A_n)= c_1 < 1$.
Find $c \in (c_1, 1)$. Then every $x\in X$ admits a decomposition $x=y_1+z_1$
with $y_1\in A_n$ and $\|z_1\|<c\|x\|$. Furthermore, $z_1 = y_2 + z_2$,
with $y_2 \in A_n$, and $\|z_2\| < c \|z_1\| < c^2 \|x\|$.
Continuing in the same way, for any $k\in \mathbb{N}$ we get a decomposition
$x=y_1+y_2+\cdots+y_k+z_k$, with $y_1,y_2,\cdots,y_k\in A_n$,
and $\|z_k\|<c^k \|x\|$.
Now, the sum $y_1+y_2+\cdots+y_k$ belongs to $A_{K^k(n)}$
(here, $K^k(n) = K(K( \ldots K(n) \ldots ))$ ($k$ times),
so that $E(x,A_{K^k(n)})\leq c^k$ for $k=0,1,2,\cdots$ and $\|x\|\leq 1$. It follows that
\begin{equation}\label{nueva}
E(x,A_{K^k(n)})\leq c^k\|x\| \text{ for }k=0,1,2,\cdots \text{ and }x\in X.
\end{equation}
Now let $\varepsilon_i = c^k$ for $K^{k-1}(n) < i \leq K^k(n)$.
For such $i$, and $x \in X$,
$$
E(x,A_i) \leq E(x, A_{K^k(n)}) \leq c^k \|x\| = \varepsilon_i \|x\| .
$$
As $\{\varepsilon_i\} \searrow 0$, this contradicts our assumption that $(X,\{A_n\})$
satisfies Shapiro's Theorem.
\end{proof}

\begin{lemma}\label{brud2}
Suppose $X$ and $(A_i)$ are as in Theorem~\ref{shapiroimpliesbrudnyi}.
Then there exists a sequence of natural numbers $s_0=0<s_1<s_2<\ldots$,
such that $(X,\{A_{s_i}\})$ satisfies the hypotheses of Brudnyi's Theorem
\ref{brud_original}.
\end{lemma}

\begin{proof}
Throughout, we are assuming that the function $K$ appearing in the definition
of an approximation scheme (Definition~\ref{def_appr_scheme}) is non-decreasing.
It suffices to select $s_0=0<s_1<s_2<\ldots$ in such a way that the sets $B_i=A_{s_i}$
satisfy (i) $B_n+B_m\subset B_{\max\{n,m\}+1}$ for all $n,m\in\mathbb{N}$, and
(ii) $E(B_{n+1}\cap S(X),B_n) \geq 1/2$ for any $n \in \N$.
Suppose $s_0=0<s_1 < \ldots < s_k$ have already been selected in such a way
that the (i) and (ii) are satisfied for $0 \leq m,n \leq k-1$. By Lemma~\ref{brud1},
$E(S(X),B_k) = 1$. As $\overline{\cup_\ell A_\ell} = X$, there exist
$\ell > K(s_k)$ and $x \in A_\ell \cap S(X)$ such that $E(x,B_k) > 1/2$.
Then $s_{k+1} = \ell$ works for us. Indeed, $E(S(X) \cap B_{k+1}, B_k) > 1/2$.
Furthermore,
$$
B_k + B_k = A_{s_k} + A_{s_k} \subset A_{K(s_k)} \subset A_\ell = B_{k+1} .
$$
Proceeding inductively, we obtain $0 = s_0 < s_1 < \ldots$
with the desired properties.
\end{proof}

\begin{proof}[Proof of Theorem~\ref{shapiroimpliesbrudnyi}]
By \cite[pp.113-114]{edwards}, there exists a convex sequence
$(\delta_n)$, convergent to $0$, such that  $\delta_n \geq \varepsilon_n$
for every $n$. By Brudnyi's theorem, there exists $x\in X$ such that
$E(x,A_{s_i})\geq \delta_i$ for $i=0,1,2,\ldots$. But $A_i\subseteq A_{s_i}$, hence
$E(x,A_i)\geq E(x,A_{s_i})\geq \varepsilon_i$ for every $i$.
\end{proof} 

\begin{corollary}\label{coro} For any approximation scheme $(X,\{A_n\})$ the following are equivalent claims:
\begin{itemize}
\item[$(a)$] $(X,\{A_n\})$  satisfies Shapiro's Theorem.
\item[$(b)$] $(X,\{A_n\})$  satisfies the weak Brudnyi's condition with constant $c$ for every $c\in (0,1]$.
\item[$(c)$] $(X,\{A_n\})$  satisfies the weak Brudnyi's condition with constant $c$ for a certain $c\in (0,1]$.
\end{itemize}
Moreover, if $X$ is a Banach space, then $(a)$, $(b)$ and $(c)$ are equivalent to:
\begin{itemize}
\item[$(d)$] For every non-decreasing sequence $\{\varepsilon_n\}_{n=0}^{\infty}\searrow 0$ there exists an element $x\in X$ such that $E(x,A_n)\geq \varepsilon_n$ for all $n\in\mathbb{N}$.
\end{itemize}
\end{corollary}

\begin{proof}
$(a)\Rightarrow (b)$ follows from Lemma~\ref{brud1}. $(b)\Rightarrow (c)$ is trivial.
To prove $(c)\Rightarrow (a)$, assume
$c\in (0,1)$ is such that $\sup_{n \in \N} E(S(X),A_n) > c > 0$.
Then for every $n\in\mathbb{N}$ there exists $x_n\in X$ with $\|x_n\|=1$ and
$E(x_n,A_{K(n)}) > c$, so
$E(x_n,A_n)\leq \|x_n\|=1\leq cE(x_n,A_{K(n)})$. This, in conjunction with Theorem \ref{maintheorem}, implies that $(X,\{A_n\})$ satisfies Shapiro's theorem.

Finally, the claim that $(a)\Rightarrow (d)$ holds for Banach spaces is just a reformulation of Theorem \ref{shapiroimpliesbrudnyi}, and $(d)\Rightarrow (a)$ is trivial.
\end{proof} 

As a consequence, we show that the approximation schemes satisfying
Shapiro's Theorem are stable under perturbations.

\begin{proposition}\label{stability}
Suppose, for a quasi-Banach space $(X, \|\cdot\|)$, there exists
$p \in (0,1]$ for which any $x_1, x_2 \in X$ satisfy
$\|x_1 + x_2\|^p \leq \|x_1\|^p + \|x_2\|^p$.
Suppose the approximation schemes $(A_n)$ and $(B_n)$ in $X$
are such that $(A_n)$ satisfies Shapiro's Theorem, and
$\liminf_n E(S(X) \cap B_n, A_n) < 1$.
Then $(X,\{B_n\})$ also satisfies Shapiro's Theorem.
\end{proposition}

\begin{proof}
Pick $C \in (\liminf_n E(S(X) \cap B_n, A_n), 1)$.
Then for any $N \in \N$ there exists $n \geq N$ such that $E(S(X) \cap B_n, A_n) < C$.
Find $0<c<1$ such that $c^p + C^{p}(1 + c^p) < 1$.
By Corollary~\ref{coro}(c), it suffices to show that, for such $n$,
$E(S(X), B_n) \geq c$, since the sequence $(E(S(X), B_k))_{k=0}^\infty$
is non-increasing.

Suppose, for the sake of contradiction, for every $x \in S(X)$ there exists $b \in B_n$
with $\|x-b\| < c$. As $b = x - (x-b)$,
$\|b\| \leq (\|x\|^p + \|x-b\|^p)^{1/p} < (1 + c^p)^{1/p}$
Then there exists $a \in A_n$ such that $\|b - a\| \leq C (1 + c^p)^{1/p}$,
hence
$$
\|x - a\|^p \leq \|b - a\|^p + \|x-b\|^p \leq c^p + C^{p} (1 + c^p) ,
$$
which contradicts Corollary~\ref{coro}(b).
\end{proof}

Another useful consequence of Corollary \ref{coro} is:
\begin{corollary}\label{increasing} Let $X$ be a quasi-Banach space and let us assume that for each $r\in\mathbb{N}$, the family $(A_{n,r})_{n=0}^{\infty}$ defines an approximation scheme in $X$ that satisfies Shapiro's Theorem and $n_1\leq n_2$, $r_1\leq r_2$ imply $A_{n_1,r_1}\subseteq A_{n_2,r_2}$. Then for every pair of increasing sequences $\{n_i\}\to\infty$, $\{r_i\}\to\infty$, the approximation scheme $(A_{n_i,r_i})$ satisfies Shapiro's theorem.
\end{corollary}
\begin{proof} Let us denote $B_i=A_{n_i,r_i}$, $i=0,1,2,\cdots$. Obviously, $(B_i)$ is an approximation scheme in $X$. By hypothesis and by Corollary \ref{coro}, for each $n,r\in\mathbb{N}$ we have that $E(S(X),A_{n,r})=1$. Hence, for each $i\in\mathbb{N}$,
we also have $E(S(X),B_i)=1$, and the result follows as a direct application of
Corollary \ref{coro}.
\end{proof}

\section{Approximation schemes that do not satisfy Shapiro's Theorem}\label{fail_shapiro}

Section~\ref{examples} below gives many examples of approximation schemes
satisfying Shapiro's Theorem. In this section, we present some examples of
schemes failing this condition, and explore their properties.

For an approximation scheme $(X,\{A_n\})$, define its {\it density sequence}
$\dens_n = \dens_n(X,\{A_n\})$ by setting, for $n \geq 0$,
$\dens_n = E(S(X),A_n)$.
Clearly, $1 = \dens_0 \geq \dens_1 \geq \ldots \geq 0$.

\begin{proposition}\label{multiplicative}
Suppose $(A_n)$ is an approximation scheme in a quasi-Banach space $X$, and
the function $L : \N \times \N \to \N$ is such that $A_m + A_n \subset A_{L(m,n)}$
for any $m,n \in \N$ (we can take $L(n,m)=K(\max\{n,m\})$).
Then $\dens_{L(m,n)} \leq \dens_m \dens_n$ for any $m,n \in \N$.
\end{proposition}

\begin{proof}
Consider $x \in X$. Fix $\delta > 0$, $m$, and $n$. Write $x = a + y$,
with $a \in A_n$, and $\|y\| \leq (1 + \delta) \dens_n \|x\|$.
Furthermore, write $y = b + z$, with $b \in A_m$ and
$$
\|z\| \leq (1 + \delta) \dens_m \|y\| \leq (1 + \delta)^2 \dens_m \dens_n \|x\| .
$$
Then $x = (a+b) + z$, with $a + b \in A_{L(m,n)}$. As $\delta > 0$ is arbitrary,
we are done.
\end{proof}

\begin{corollary}
Let $(A_n)$ be an approximation scheme in a quasi-Banach space $X$.
Then $(A_n)$ satisfies Shapiro's Theorem if
and only if $\dens_n=1$ for any $n\in\mathbb{N}$.
\end{corollary}

As a particular case of Proposition~\ref{multiplicative},
consider an approximation scheme arising from a dictionary.
We say that a set ${\mathcal{D}}$ is a
{\it dictionary} in a quasi-Banach space $X$ if
$\overline{\mathbf{span}[\mathcal{D}]} = X$. Define the approximation scheme $(X,\Sigma_n(\mathcal{D}))$ by setting
\begin{equation}
\Sigma_0(\mathcal{D})=\{0\}; \ \
\Sigma_n({\mathcal{D}}) = \bigcup_{F \subset {\mathcal{D}}, |F| \leq n} \spn[F]
\, \, \, {\mathrm{for}} \, \, n \geq 1 .
\label{nterm}
\end{equation}
Then $\Sigma_n({\mathcal{D}}) + \Sigma_m({\mathcal{D}}) =
\Sigma_{n+m}({\mathcal{D}})$ for every $n, m \geq 0$ (hence
we can take $L(m,n) = m+n$). We thus have:

\begin{corollary}\label{n_term-decay}
Suppose an approximation scheme $(\Sigma_n({\mathcal{D}}))$ is constructed as described
in the previous paragraph. Then $\dens_{m+n} \leq \dens_m \dens_n$
for any $m$ and $n$. In particular, if $\dens_m < 1$ for some $m$,
then the sequence $(\dens_n)$ decays exponentially or faster.
\end{corollary}


In Section~\ref{examples}, we shall see many dictionaries (some quite
redundant) for which $\dens_n = 1$ for any $n$.
These dictionaries cannot be ``too redundant.''
Indeed, if a dictionary ${\mathcal{D}}$ is a $c$-net of the unit sphere $S(X)$
for some $c < 1$, then $\dens_1 \leq c$, hence $\dens_n \leq c^n$ for every $n$.

Below we consider an ``extreme'' case of $\dens_n$
becoming $0$ for $n$ large enough.

\begin{proposition}\label{general}Let $(X,\{A_n\})$ be an approximation scheme.
The following are equivalent:
\begin{itemize}
\item[$(a)$] $\bigcup \overline{A_n} = X$ (equivalently, for all $x\in X$ there exists $n=n(x)\in \N$ such that $E(x,A_n)=0$).
\item[$(b)$] $\overline{A_n}=X$ for some $n \in \N$ (equivalently, $E(x,A_{n})=0$
for all $x\in X$).
\end{itemize}
Consequently, $\bigcup \overline{A_n}\neq X$ if and only if $\dens_n>0$ for all $n$.
\end{proposition}

\begin{proof}
The implication (b) $\Rightarrow$ (a) is obvious. To prove the converse,
suppose $X = \cup_n \overline{A_n}$. By Baire Category Theorem,
for some $n$, there exist $x \in X$ and $c > 0$ such that
$B(x,c)$ (the ball with the center at $x$, and radius $c$)
lies inside of $\overline{A_n}$. By symmetry, $B(-x,c) \subset \overline{A_n}$.
Then
$$
B(0,c) \subset B(x,c) + B(-x,c) \subset \overline{A_n} + \overline{A_n} \subset
\overline{A_{K(n)}} .
$$
But $\lambda \overline{A_{K(n)}} = \overline{A_{K(n)}}$ for any scalar $\lambda$ and $m$,
hence $\overline{A_{K(n)}} = X$.

To prove the last claim of the Proposition, note that
$X \neq \overline{A_n}$ if and only if $\dens_n > 0$.
\end{proof}


\begin{corollary}\label{hamel}
Suppose ${\mathcal{D}}$ is a Hamel basis in a Banach space $X$.
Then there exists $n \in \N$ for which $\Sigma_n({\mathcal{D}})$
is dense in $X$.
\end{corollary}

Note that there are no uniform bounds for the values of $n$ with the
property outlined in Propositions~\ref{general}(b) and Corollary~\ref{hamel}.
Indeed, by \cite{BDHMP}, any Banach space has a dense Hamel basis ${\mathcal{D}}$.
In particular, $\Sigma_1({\mathcal{D}})$ is dense in $X$.
On the other hand, consider
a space $X = \ell_\infty^N \oplus_p Y$. If ${\mathcal{H}}$ is a
Hamel basis in $Y$, then ${\mathcal{D}} = \{e_i \oplus h : 1 \leq i \leq N ,
h \in {\mathcal{H}}\}$ ($(e_i)$ is the canonical basis in $\ell_\infty^N$)
is a Hamel basis in $X$. Then, for any $n < N$, there exists a norm $1$
$x \in X$ such that $E(x, \Sigma_n({\mathcal{D}})) = 1$ (indeed,
$(e_1 + \ldots + e_N) \oplus 0$ has this property).

Another corollary deals with Hamel bases indexed by positive reals.

\begin{corollary}\label{real_line}
Suppose $\mathcal{D}=\{e_i\}_{i\in [0,\infty)}$ is a Hamel basis of a separable Banach space
$X$, and $A_n=\spn[\{e_i\}_{i\leq n}]$. Then there exists $n_0\in\N$ such that
$A_{n_0}$ is dense in $X$. In particular, $A_{n_0}$ is an infinite codimensional dense subspace of $X$.
\end{corollary}

Next, we present an example where the ``slowest possible'' rate of approximation $E(x,A_n)$
is precisely controlled.

\begin{theorem}\label{control_approximation}
Suppose $X$ is $L_\infty(0,1)$, $\ell_\infty$, or $C(\Delta)$
(where $\Delta$ is the ternary Cantor set). Suppose, furthermore,
that $1 \geq \varepsilon_1 \geq \varepsilon_2 \geq \ldots \geq 0$,
and $\lim_n \varepsilon_n = 0$. Then there exists an
approximation scheme $(A_n)$ in $X$ such that
the $\dens_n \leq \varepsilon_n$ for any $n$, and
there exists $x \in S(X)$ with the property that
$E(x,A_n) \geq \varepsilon_n/(1 + \varepsilon_n)\geq \frac{\varepsilon_n}{2}$ for any $n$.
\end{theorem}

The above theorem is stated for real Banach spaces.
Similar results (with different constants) can also be obtained
in the complex case.

\begin{proof}
We start by presenting the construction of $(A_n)$ in the case of
$X = L_\infty(0,1)$. Find a sequence of positive integers
$m(1) \leq m(2) \leq \ldots$, such that, for any $n$,
$1/m(n) \leq \varepsilon_n \leq 1/(m(n)-1)$.
Define $A_n$ as the set of (equivalence classes of) functions
in $L_\infty(0,1)$ assuming no more than $m(n)$ different values.
In other words, $A_n$ consists of all functions
$a = \sum_{i=1}^{m(n)} \alpha_i \chi_{E_i}$, where
$(E_i)_{i=1}^{m(n)}$ is a partition of $(0,1)$ into
measurable sets.

(1) For a norm $1$ function $x \in L_\infty(0,1)$ and $n \in \N$, we shall
find $a \in A_n$ such that $\|x - a\| \leq 1/m(n)$. To this end,
let $s_j = (2j - 1)/m(n) - 1$ ($1 \leq j \leq m(n)$).
Let $I_1 = [-1,-1+2/m(n)]$, and $I_j = (-1 + 2(j-1)/m(n), -1 + 2j/m(n)]$
for $2 \leq j \leq m(n)$. Note that $s_j$ is the midpoint of $I_j$.
For $t \in (0,1)$, define $a(t) = s_j$ if $x(t) \in I_j$. Then $a$
is defined almost everywhere, $a \in A_n$, and
$\|x-a\| \leq 1/m(n) \leq \varepsilon_n$.

(2) We claim that the function $x(t) = 2t-1$ is such that $\|x-a\| \geq
1/m(n) \geq \varepsilon_n/(1+\varepsilon_n)$. Indeed, suppose
$a$ takes values $a_1 < a_2 < \ldots a_k$, with $k \leq m(n)$,
and $\|x - a\| = c < 1/m(n)$. Then
$x(t) \in \cup_{j=1}^k [a_j - c, a_j + c]$ almost everywhere, which,
in turn, implies $a_1 \leq -1 + c$, $a_j + 2c \geq a_{j+1}$ for
$1 \leq j \leq k-1$, and $a_k \geq 1 - c$. This, however, is impossible.

The case of $X = \ell_\infty$ is handled the same way, with
minor modifications. For $X = C(\Delta)$, consider {\it elementary
intervals} $T_{s,k} = [\sum_{j=1}^k s_j 3^{-j}, \sum_{j=1}^k s_j 3^{-j} + 3^{-k}]$
($k \in \N$, $s = (s_1, \ldots, s_k) \in \{0,2\}^k$).
Define $A_n$ to be the set of functions $a$ on $\Delta$ such that
(i) $a$ attains no more than $m(n)$ different values, and (ii) there
exists $k \in \N$ such that the restriction of $a$ to $T_{s,k} \cap \delta$
is constant for any $s \in \{0,2\}^k$. To show $E(x,a) \leq \|x\|/m(n)$
for any $x \in C(\Delta)$, take into account the uniform continuity of $x$.
A version of the ``Cantor ladder'' gives an example of $x$ with
$E(x,A_n) \geq 1/m(n) \geq \varepsilon_n/(1+\varepsilon_n)$ for any $n$.
\end{proof}

The theorem above implies that many Banach spaces contain an approximation
scheme with controlled rate of approximation.

\begin{corollary}\label{many_spaces_controlled}
Suppose $X$ is an infinite dimensional Banach space, and
either (1) $X$ is injective, or (2) $X$ is separable, and
contains an isomorphic copy of $C(\Delta)$.
Then there exists a constant $c > 0$ such that, for every sequence
$1 \geq \varepsilon_1 \geq \varepsilon_2 \geq \ldots \geq 0$,
satisfying $\lim_n \varepsilon_n = 0$, there exists an approximation
scheme $(A_n)$ with the property that $\dens_n \leq \varepsilon_n$ for any $n$,
and there exists $x \in S(X)$ with the property that
$E(x,A_n) \geq c \varepsilon_n$ for all $n$.
\end{corollary}

\begin{proof}
(1) Suppose $X$ is injective. Then (see \cite[Theorem 2.f.3]{LT1}),
there exists a subspace $Y$ of $X$, a projection $P$ from $X$ onto $Y$,
and an isomorphism $U : Y\to \ell_\infty$ with contractive inverse.
By Theorem~\ref{control_approximation}, there exists an approximation
scheme $(B_n)$  in  $\ell_{\infty}$ such that $E(z,B_n) \leq \delta_n \|z\|$
for any $n$ and $z \in \ell_\infty$, where $\delta_n = \varepsilon_n/(\|U\| \|P\|)$.
Furthermore, there exists $z_0 \in \ell_\infty$ with $\|z_0\| = 1$, and
$E(z_0, B_n) \geq \delta_n/2$ for any $n$. We claim that the family
$A_n = \ker P+\spn[U^{-1}(B_n)]$ has the desired properties.

Note first that, for any $x \in X$,
$$
E(x,A_n) \leq E(Px, U^{-1}(B_n)) \leq E(UPx, B_n) \leq \delta_n \|UPx\| \leq
\delta_n \|U\| \|P\| \|x\| = \varepsilon_n \|x\| .
$$
On the other hand, find $z_0 \in S(\ell_\infty)$ such that $E(z_0, B_n) \geq
\delta_n/2$ for any $n$. Then $x_0 = U^{-1} z_0$ has norm not
exceeding $1$. To estimate $E(x_0, A_n)$, consider $b \in A_n$.
Then
$$
\|x_0 - b\| \geq \frac{1}{\|P\|} \|P(x_0 - b)\| = \frac{1}{\|P\|} \|x_0 - P b\| .
$$
Furthermore,
$$
\|x_0 - P b\| \geq \frac{1}{\|U\|} \|U x_0 - U P b\| =
\frac{1}{\|U\|} \|z_0 - U P b\| \geq
\frac{1}{\|U\|} E(z_0, B_n) \geq \frac{\varepsilon_n}{2\|U\|} .
$$
This leads to the desired estimates on $E(x_0, A_n)$.

The proof of (2) is very similar, except that now, we rely on the fact that
any separable Banach space containing a copy of $C(\Delta)$, must also contain
a complemented copy of the latter space (see e.g.~\cite{rosenthal}).
\end{proof}

\begin{remark}\label{control_appr_c_0}
A weaker version of Theorem~\ref{control_approximation}
holds in the space $c_0$. More precisely, suppose
$1 \geq \varepsilon_1 \geq \varepsilon_2 \geq \ldots \geq 0$,
and $\lim_n \varepsilon_n = 0$. Then there exists an approximation
scheme $(A_n)$ in $c_0$, with the following properties:
\begin{enumerate}
\item
$\varepsilon_n \geq \dens_n \geq \varepsilon_n/3$.
\item
For any non-increasing sequence $\{\delta_n\} \in c_0$ there exists $x \in c_0$ such that
$E(x,A_n) \geq \delta_n \varepsilon_n$ for every $n$.
\end{enumerate}
As the construction is similar to the one presented above,
we do not describe it here.
\end{remark}

\section{Connection with Central Theorems of Approximation Theory} \label{seccion_central}

In this section we examine the connections between the so called central theorems of approximation theory -- that is, the classical Jackson's (direct) and Bernstein's (inverse)
results for the speed of approximation by a given approximation scheme -- and Shapiro's
Theorem.

\begin{definition} Let $(X,\{A_n\})$ be an approximation scheme and let $Y$ be a quasi-semi-Banach space continuously and strictly included in the
quasi-Banach space $X$. We say that the approximation scheme $(X,\{A_n\})$ satisfies
({\it generalized}) {\it Jackson's Inequality} with respect to $Y$ if
there exists a sequence $(c_n)$ such that 
$\lim_{n\to\infty}c_n=+\infty$ and
\begin{equation}\label{J}
E(x,A_n)\leq \frac{1}{c_n}\|x\|_Y \text{ for all }x\in Y.
\end{equation}
The approximation scheme $(X,\{A_n\})$ is said to satisfy ({\it generalized})
{\it Bernstein's Inequality} with respect to $Y$ if
$\bigcup_{n=0}^{\infty}A_n\subseteq Y$, and  there exists a sequence $(b_n)$
such that 
$\lim_{n\to\infty}b_n=+\infty$ and
\begin{equation}\label{B}
\|x_n\|_Y\leq b_n \|x_n\|_X \text{ for all }x_n\in A_n.
\end{equation}
(The classical definition for these inequalities appears when $b_n=c_n=Cn^r$).
\end{definition}

Jackson's Inequality does not imply Shapiro's Theorem.
On the contrary, Jackson's Inequality is satisfied
for a sufficiently large space $Y\subset X$ if and only if Shapiro's Theorem fails.


\begin{proposition}\label{shapiro_jackson} For an approximation scheme  $(X,\{A_n\})$, the following are equivalent:
\begin{itemize}
\item[$(i)$] $(X,\{A_n\})$ does not satisfy Shapiro's Theorem.
\item[$(ii)$] $(A_n)$ satisfies Jackson's Inequality for some finite codimensional subspace 
$Y \subset X$.
\item[$(iii)$] $(A_n)$ satisfies Jackson's Inequality for every subspace 
$Y \subset X$.
\end{itemize}
 In particular, if $(X,\{A_n\})$ satisfies  Shapiro's Theorem and $Y$ is a quasi-normed subspace of $X$ such that $(A_n)$ satisfies Jackson's Inequality with respect to $Y$, then $Y$ must be of infinite codimension.
\end{proposition}

\begin{proof}
$(iii) \Rightarrow (ii)$ is trivial.

$(i) \Rightarrow (iii)$: by Corollary~\ref{coro}, there exists a sequence
$\{\varepsilon_n\} \searrow 0$ such that  $E(x,A_n) \leq \varepsilon_n \|x\|_X$
for any $x \in X$. The space $Y$ is continuously embedded into $X$, hence
there exists a constant $C$ such that  $\|y\|_X \leq C \|y\|_Y$ for any $y \in Y$.
Therefore, $E(y,A_n) \leq C \varepsilon_n \|y\|_Y$ for any $y \in Y$,
which is \eqref{J} with $c_n = (C \varepsilon_n)^{-1}$.

$(ii) \Rightarrow (i)$: by Corollary~\ref{coro},
it suffices to find $m \in \N$ for which $E(S(X),A_m) < 1$.
Let $N = \dim X/\overline{Y}$.
If $N = 0$ (that is, $X = \overline{Y}$), there is nothing to prove.
Otherwise, consider the quotient map $q : X \to E = X /\overline{Y}$.
Find $x_1, \ldots, x_N$ in $X$, such that the vectors $e_i = qx_i$
form a normalized basis in $E$, and $\|x_i\| < 2$ for every $i$. Then there exists
a constant $C_1 \geq 1$ such that $C_1^{-1} \max_{1 \leq i \leq N} |\alpha_i| \leq
\|\sum_{1 \leq i \leq N} \alpha_i q x_i\|$ for any $N$-tuple of scalars $(\alpha_i)$.

Recall the existence of a constant $C_X \geq 1$ such that $\|x+y\| \leq C_X(\|x\| + \|y\|)$ for any $x, y \in X$.
By induction,
\begin{equation}
\|\sum_{j=1}^m z_j\| \leq C_X^{m-1} \sum_{j=1}^m \|z_j\|
\label{eq:sum_of_m}
\end{equation}
for any $z_1, \ldots, z_m \in X$.
We claim that any $x \in S(X)$ has a representation
\begin{equation}
x = y + \sum_{i=1}^N \alpha_i x_i, \, \,
{\mathrm{with}} \, \, \max_{1 \leq i \leq N} |\alpha_i| \leq C_1 ,
\, \, {\mathrm{and}} \, \, \|y\| \leq C_2 = 2 C_1 C_X^N .
\label{eq:represent}
\end{equation}
Indeed, $\|qx\| \leq 1$, hence one can write $qx = \sum_{i=1}^N \alpha_i e_i$,
with $(\alpha_i)$ as above. Then $y = x - \sum_{i=1}^N \alpha_i x_i \in Y$,
and \eqref{eq:sum_of_m} yields the desired estimate on the norm of $y$.

Pick $c \in (0,1)$, and show the existence of $m \in \N$ for which
$E(S(X),A_m) < c$. Start by using \eqref{J} to find $n \in \N$ such that
$E(y, A_n) \leq c \|y\|/(2 N C_2 C_X^N)$ holds for every $y \in Y$.
Then find $k \geq n$ such that, for every $i \in \{1, \ldots, N\}$, there exists
$a_i \in A_k$ satisfying $\|x_i - a_i\| < c/(2 N C_1 C_X^N)$.
We claim that $E(S(X),A_m) \leq c$, where $m = K(K(\ldots(k)\ldots))$ ($N+1$ times).
Indeed, any $x \in S(X)$ can be represented as in \eqref{eq:represent}.
Find $a_0 \in A_k$ satisfying
$\|y-a_0\| < c \|y\|/(2 N C_2 C_X^N) \leq c/(2 N C_X^N)$. Then
$a = a_0 + \sum_{i=1}^N \alpha_i a_i \in A_m$, and, by \eqref{eq:sum_of_m},
\begin{align*}
\|x - a\|
&
=
\|(y + \sum_{i=1}^N \alpha_i x_i) - (a_0 + \sum_{i=1}^N \alpha_i a_i)\| \leq
C_X^N \big(\|y - a_0\| + \sum_{i=1}^N |\alpha_i| \|a_i\|\big)
\\
&
<
C_X^N \Big( \frac{c}{2 N C_X^N} + N C_1 \frac{c}{2 N C_1 C_X^N} \Big) < c .
\end{align*}
Thus, $E(S(X),A_m) \leq c < 1$. An application of Corollary~\ref{coro}
completes the proof.
\end{proof}


Now we concentrate on Bernstein's Inequality.

\begin{theorem}\label{central} Let $(X,\{A_n\})$ be an approximation scheme that
satisfies Bernstein's Inequality for a certain proper subspace Y of X.
Then $(X,\{A_n\})$ satisfies Shapiro's Theorem.
\end{theorem}

\begin{proof} 
We show that if the approximation scheme does not satisfy Shapiro's Theorem,
and it satisfies Bernstein's Inequality with respect to a certain space $Y$,
then the norms of $Y$ and $X$ are equivalent. So, let us assume that $(X,\{A_n\})$
satisfies \eqref{B} for a certain sequence of positive real numbers $(b_n)$,
and a quasi-Banach space $Y$ continuously included in $X$.
By renorming $X$ and $Y$ if necessary (see Section~\ref{shapiro}),
we may assume the existence
of $p_X, p_Y \in (0,1]$ such that:
(i) for any $x_1, x_2 \in X$,
$\|x_1 + x_2\|_X^{p_X} \leq \|x_1\|_X^{p_X} + \|x_2\|_X^{p_X}$,
(ii) for any $y_1, y_2 \in Y$,
$\|y_1 + y_2\|_Y^{p_Y} \leq \|y_1\|_Y^{p_Y} + \|y_2\|_Y^{p_Y}$.
Letting $p = \min \{p_X, p_Y\}$,
we see that, for any $x_1, x_2 \in X$ and $y_1, y_2 \in Y$,
$\|x_1 + x_2\|_X^p \leq \|x_1\|_X^p + \|x_2\|_X^p$, and
$\|y_1 + y_2\|_Y^p \leq \|y_1\|_Y^p + \|y_2\|_Y^p$.
For the sake of brevity, we shall denote $\| \cdot \|_X$
simply by $\| \cdot \|$.

If $(X,\{A_n\})$ does not satisfy Shapiro's Theorem, Corollary \ref{coro} guarantees
the existence of $n_0 \in \N$ for which $E(S(X),A_{n_0}) < (1/2)^{1/p}$.
Therefore, for any $x \in X$, there exist $a \in A_{n_0}$ and $x^\prime \in X$
such that $x = a + x^\prime$, and $\|x^\prime\| < 2^{-1/p} \|x\|$.

Now pick $x \in B(X) \backslash Y$. By the above, we can find $a_0 \in A_{n_0}$
and $x_0 \in X$ such that  $x = a_0 + x_0$, with $\|x_0\| < 2^{-1/p}$.
Furthermore, we can write $x_0 = a_1 + x_1$, with $a_1 \in A_{n_0}$,
and $\|x_1\| < 2^{-2/p}$. Proceeding further in the same manner,
we write, for each $m$, $x = a_0 + a_1 + \ldots + a_m + x_m$,
with $a_0, a_1, \ldots \in A_{n_0}$, and $\|x_m\| < 2^{-m/p}$.
Note that $a_m = x_{m-1} - x_m$, hence
$\|a_m\| \leq (\|x_{m-1}\|^p + \|x_m\|^p )^{1/p} < 3^{1/p} 2^{-m/p}$.

Let $z_m = x - x_m$. As $\lim_m \|x_m\| = 0$, the sequence $(z_m)$ converges
to $x$ in the space $X$. We shall show that $(z_m)$ is a Cauchy sequence in $Y$.
Indeed, for $n > m$, $z_n - z_m = \sum_{k=m+1}^n a_k$. Furthermore,
$\|a_k\|_Y \leq b_{n_0} \|a_k\|$, for each $k$. Therefore,
$$
\|z_n - z_m\|_Y^p \leq \sum_{k=m+1}^n \|a_k\|_Y^p
\leq b_{n_0}^p \sum_{k=m+1}^n \|a_k\|^p <
3 b_{n_0}^p \sum_{k=m+1}^n 2^{-(k+1)} < 3 b_{n_0}^p 2^{-m} .
$$
As $Y$ is a subset of $X$, the sequence $(z_m)$ must
converge to $x$ in the space $Y$. This leads to a contradiction,
since $x$ was selected in such a way that $x \notin Y$.
\end{proof} 

As a corollary, we conclude that the property of satisfying Shapiro's Theorem is,
under certain conditions, inherited by subspaces.

\begin{corollary} Suppose the approximation scheme $(X,\{A_n\})$ satisfies
Bernstein's inequality for a proper subspace $Y$ of $X$. Suppose, furthermore,
that $Z$ is another quasi-normed subspace of $X$, properly containing $Y$, and
such that $\bigcup_{n=0}^{\infty} A_n$ is dense in $Z$
(in the topology determined by the norm of $Z$). Then
$(Z,\{A_n\})$ satisfies Shapiro's Theorem.
\end{corollary}


Section~\ref{examples} contains some examples where the fact that a given
approximation scheme satisfies Shapiro's Theorem is deduced from a
Bernstein's Inequality.


Below we introduce the so called ``smoothness spaces''
(or ``abstract approximation spaces'').
If $(A_n)$ is an approximation scheme in $X$, we define, for
$0 < q \leq \infty$ and $0 < r < \infty$,
\begin{equation}\label{approx_space_def}
A_q^r = A_q^r(X, \{A_n\}) =
\{x\in X:|x|_{A_q^r}=\|\{(n+1)^{r-1/q}E(x,A_n)\}\|_{\ell^q}<\infty\} .
\end{equation}
If $A_n+A_n\subseteq A_{cn}$ for a constant $c > 1$, then $A_q^r$ is a quasi-Banach space
\cite{almiraluther2}.
It was shown by DeVore and Popov (see \cite[Th. 9.3, p. 236]{devore}) that
$A_q^r$ satisfies Bernstein's Inequality: $|x|_{A_q^r} \leq Cn^r\|x\|_X$ for all $x\in A_n$.

To apply Theorem \ref{central} with $Y=A_q^r$, we need $Y$ to be a proper subspace of $X$,
which does not always hold.
For instance, suppose ${\mathcal{D}}$ is a dictionary in a Banach space $X$,
which is $1/2$-dense in $X$. Let $A_n = \Sigma_n({\mathcal{D}})$.
Clearly, $A_n + A_n \subset A_{2n}$. By Corollary~\ref{n_term-decay} and
the discussion following it, $E(x, A_n) \leq 2^{-n} \|x\|$ for any $x \in X$.
Therefore, $A_q^r=X$ for any $q,r$ (with equivalent norms).


On the other hand, there are many classical results in Approximation Theory devoted to the characterization of the approximation spaces $A_q^r$ as smoothness spaces of functions (Besov, etc.), and these are always proper subspaces of the ground space $X$.
In this setting, one can apply Theorem~\ref{central} to show that
the corresponding approximation scheme satisfies Shapiro's Theorem.
The same applies to the situation when $X$ is a space of operators, and
membership in $A_q^r$ reflects the ``degree of compactness''
(see e.g.~\cite{Pie}).

Below, we show that the spaces $A_q^r$ form a scale of subspaces
of $X$ if the approximation scheme $(A_n)$ satisfies
Shapiro's Theorem. We also present other results on the spaces $A_q^r$.

\begin{corollary} \label{coro2}
Let $(X,\{A_n\})$ be an approximation scheme such that $A_n+A_n\subseteq A_{cn}$
for a certain constant $c>1$. Then the following are equivalent:
 \begin{itemize}
 \item[$(a)$] $(X,\{A_n\})$ satisfies Bernstein's Inequality for
some proper subspace $Y$ of $X$.
 \item[$(b)$] $(X,\{A_n\})$ satisfies Shapiro's theorem.
 \item[$(c)$] For every $r > 0$ there exists $x\in X$ such that
$E(x,A_n)\not=\mathbf{O}(n^{-r})$.
 \item[$(d)$] For a certain $r > 0$, there exists $x\in X$ such that $E(x,A_n)\not=\mathbf{O}(n^{-r})$.
 \item[$(e)$] For any $q \in (0,\infty]$ and $r \in (0,\infty)$,
$A_q^r$ is a proper subspace of $X$.
 \item[$(f)$] For some $q \in (0,\infty]$ and $r \in (0,\infty)$,
$A_q^r$ is a proper subspace of $X$.
 \end{itemize}
Moreover, if any of these conditions is satisfied, then  for every $q, r > 0$, $A_q^r$ is an infinite codimensional subspace of $X$.
 \end{corollary}

\begin{proof} 
The implication $(a)\Rightarrow (b)$ is a reformulation of Theorem \ref{central}.
$(b) \Rightarrow (c) \Rightarrow (d)$ and $(c) \Rightarrow (e) \Rightarrow (f)$
are trivial.

$(d) \Rightarrow (a)$: If $E(x,A_n)\not=\mathbf{O}(n^{-r})$ for some $x\in X$ and
$r > 0$, then $x\not\in A_{\infty}^r$.
Then $A_{\infty}^r$ is strictly contained in $X$, and $(X,\{A_n\})$
satisfies Bernstein's inequality for $Y=A_{\infty}^r$.
Theorem~\ref{central} yields $(a)$.

$(f) \Rightarrow (c)$: consider $x \in X \backslash A_q^r$.
By \eqref{approx_space_def}, $x \notin A_\infty^s$ for any $s > r$.

The last claim follows from the fact that $(A_n)$ satisfies Jackson's inequality \eqref{J}
with $Y=A_q^r$, and Proposition \ref{shapiro_jackson}.
\end{proof} 

To further investigate abstract approximation spaces, denote by $B_q^r$ the
closure of $\bigcup_{n=0}^{\infty}A_n$ in $A_q^r$. Clearly, $B_q^r$ is a
closed subspace of $A_q^r$ (with the same norm). If $A_n + A_n \subset A_{cn}$
with some $c$, the results of \cite[Section 3]{almiraluther2} imply that
$B_q^r = A_q^r$ for $0 < q < \infty$, and that
$B_\infty^r=\{x\in A_\infty^r:\lim_{n\to\infty}(n+1)^r E(x,A_n)=0\}$.

\begin{proposition}\label{approx_scale}
Suppose the approximation scheme $(X,\{A_n\})$ satisfies Shapiro's Theorem,
and $A_n + A_n \subset A_{cn}$ for some $c$.
Then $(B_q^r, \{A_n\})$ satisfies Shapiro's Theorem.
Consequently, $A_q^{r+\varepsilon}(X,\{A_n\})$ is an infinite codimensional
subspace of $A_u^r(X,\{A_n\})$ for all $\varepsilon>0$, all $0<q\leq \infty$ and all $0<r,u<\infty$.
\end{proposition}

\begin{proof}
A small modification of the proof of $(a)\Rightarrow (b)$ in Theorem \ref{maintheorem} shows that there exists a constant $C>1$ and a sequence $\{x_n\}_{n\in\mathbb{N}_0}$ (where $\mathbb{N}_0\subseteq \mathbb{N}$ is an infinite sequence) such that $E(x_n,A_n)\leq C E(x_n,A_{K^2(n)})$ for all $n\in\mathbb{N}_0$. By the density of $\bigcup_nA_n$ in $X$, we can find another sequence
$\{a_n\}_{n\in\mathbb{N}_0}\subset \bigcup_n A_n\subset B_q^r$ such that $E(a_n,A_n)\leq C E(a_n,A_{K^2(n)})$ for all $n\in\mathbb{N}_0$. Hence for every $n\in\mathbb{N}_0$ and $m\in\{n,n+1,\cdots,K^2(n)\}$ we have $E(a_n,A_m)\leq E(a_n,A_n) \leq CE(a_n,A_{K^2(n)})$.
Furthermore, by Lemma 3.16 from \cite{almiraluther2}, there exist $A,B>0$
(depending only on $X$ and the parameters $q,r$) such that, for every $n\in\mathbb{N}$,
\[
A\left\|\{(k+1)^{r-\frac{1}{q}}E(a_n,A_{\max\{k,K(n)\}})\}\right\|_{\ell_q}\leq E(a_n,A_n)_{A_q^r} \leq
B \left\|\{(k+1)^{r-\frac{1}{q}}E(a_n,A_{\max\{k,n\}})\}\right\|_{\ell_q} .
\]
Therefore,
\begin{eqnarray*}
E(a_n,A_n)_{B_q^r} &\leq& B \left\|\{(k+1)^{r-\frac{1}{q}}E(a_n,A_{\max\{k,n\}})\}\right\|_{\ell_q} , \\
E(a_n,A_{K(n)})_{B_q^r} &\geq& A\left\|\{(k+1)^{r-\frac{1}{q}}E(a_n,A_{\max\{k,K^2(n)\}})\}\right\|_{\ell_q} .
\end{eqnarray*}
It follows that
\begin{eqnarray*}
E(a_n,A_n)_{B_q^r} &\leq& B\left\|\{(k+1)^{r-\frac{1}{q}}E(a_n,A_{\max\{k,n\}})\}\right\|_{\ell_q} \\
&\leq & B\left(\sum_{k=0}^{K^2(n)}(k+1)^{rq-1}C^qE(a_n,A_{K^2(n)})^q+ \sum_{k=K^2(n)+1}^{\infty}(k+1)^{rq-1}E(a_n,A_k)^q \right)^{\frac{1}{q}} \\
&\leq& CBA^{-1} E(a_n,A_{K(n)})_{B_q^r} .
\end{eqnarray*}
By Theorem~\ref{maintheorem}$(b)\Rightarrow (a)$,
$(B_q^r, \{A_n\})$ satisfies Shapiro's Theorem.

To prove the second part of our proposition, recall
the reiteration theorem: if an approximation scheme satisfies $A_n+A_n\subseteq A_{cn}$
for a certain constant $c$, then
$$A_q^{r_2}(A_s^{r_1}(X,\{A_n\}),\{A_n\})=A_q^{r_1+r_2}(X,\{A_n\})$$ (this is proved in
\cite{Pie} for the particular case of $A_n+A_m\subseteq A_{n+m}$, and
in \cite[Example 3.36]{almiraluther2} in full generality). Hence,
$$
A_q^{r+\varepsilon}(X,\{A_n\})=A_q^{\varepsilon}(A_u^{r}(X,\{A_n\}),\{A_n\})=
A_q^{\varepsilon}(A_u^{r},\{A_n\}) .
$$
As $u < \infty$,
the first part of our proposition shows that $(A_u^{r},\{A_n\})$ satisfies Shapiro's Theorem.
By Corollary~\ref{coro2}, 
$A_q^{\varepsilon}(A_u^r,\{A_n\})$  is an infinite codimensional subspace of
$A_u^{r}(X,\{A_n\})$.
\end{proof}

Finally, another consequence of Theorem \ref{central} is the following
\begin{corollary}
Suppose $(X,\{A_n\})$ is an approximation scheme, such that, for every $n \in \N$,
$A_n+A_n\subseteq A_{cn}$ ($c>1$ is independent of $n$), and $A_n$
is boundedly compact in $X$ (that is, any bounded subset of $A_n$ is relatively
compact in $X$). Then $(X,\{A_n\})$ satisfies Shapiro's Theorem.
 \end{corollary}

\begin{proof}
If each $A_n$ is boundedly compact in $X$ then for every $r>0$ the natural inclusion $A_{\infty}^r\hookrightarrow X$ is a compact operator
(see \cite[Theor. 3.32]{almiraluther2}).
In particular, $A_{\infty}^r$ is strictly contained in $X$,
and we can apply Theorem \ref{central} with $Y=A_{\infty}^r$.
\end{proof}

\section{Examples of schemes satisfying Shapiro's Theorem}\label{examples}

In this section, we present a collection of examples of approximation schemes
satisfying Shapiro's Theorem. The main tools involved are (i) Property (P), (ii) Bernstein's Inequality  and (iii) the characterization of approximation schemes satisfying Shapiro's Theorem given in Corollary \ref{coro}. Many examples involve the order of the
best $n$-term approximation with respect to a dictionary.

\subsection{Biorthogonal systems and their generalizations}\label{examp_biorth}
Suppose $X$ is a quasi-Banach space, $I$ is an infinite index set,
and $(X_i)_{i \in I}$ are non-trivial subspaces of $X$. We say that $(X_i)$ form a
{\it complete minimal bounded decomposition} of $X$
({\it CMBD}, for short) if $X = \overline{\spn[X_i : i \in I]}$, and, for every $i \in I$,
there exists $x \in X_{i}$ such that $E(x, \spn[X_j : j \neq i]) > c \|x\|$
($c > 0$ is independent of $i$).

A CMBD can be regarded as a generalization of 
a complete minimal system. Recall that a family $(x_i)_{i \in I}$ in a
Banach space $X$ is called {\it minimal} if, for any $i \in I$,
$x_i$ doesn't belong to the closure of $\spn[x_j : j \in I \backslash \{i\}]$.
A minimal system is called {\it complete} if $\spn[x_i : i \in I]$ is dense in $X$.
It is easy to see that a minimal system gives rise to a
{\it biorthogonal system} $(x_i, f_i)$, where $x_i \in X$, $f_i \in X^*$, and
$\langle f_i, x_j \rangle = \delta_{ij}$ (Kronecker's delta).
A biorthogonal system is {\it bounded} if $\sup_i \|x_i\| \|f_i\| < \infty$.


It is easy to see that, if $(x_i, f_i)$ is a bounded complete biorthogonal system, then
the family of spaces $X_i = \spn[x_i]$ forms a CMDB.
It is known that every separable Banach space has a complete bounded biorthogonal
system $(x_i, f_i)_{i \in I}$ such that $\cap_{i \in I} \ker f_i = \{0\}$
\cite[Theorem 1.27]{HMVZ}.
Certain non-separable spaces also possess complete bounded biorthogonal systems
(see e.g.~Sections 4.2 and 5.2 of \cite{HMVZ}).

In addition to biorthogonal systems, CMBDs arise when one considers
a dictionary consisting of two or more bases, possessing certain
``mutual coherence.'' Several examples can be found in
Section 4 of \cite{gribonvaltres}. For instance, the union of Haar and
Walsh bases works very nicely.

The following two theorems show that the approximation schemes
arising from CMBDs or biorthogonal systems have Property (P).
Furthermore, as the approximation schemes described there
satisfy $A_n + A_n \subset A_{2n}$, both schemes satisfy Shapiro' Theorem.

\begin{theorem}\label{decomposition}
Consider a quasi-Banach space $X$ such that, for a certain fixed $p>0$ and for any $x_1, \ldots, x_m \in X$,
$$
\|x_1 + \ldots + x_m\|^p \leq C^p (\|x_1\|^p + \ldots + \|x_m\|^p) .
$$
Suppose $(X_i)_{i \in I}$ is a
complete minimal bounded decomposition of $X$, with
$E(x, \spn[X_j : j \neq i]) \geq c \|x\|$ for any $i \in I$, and $x \in X_i$.
Suppose, furthermore, that $E$ is a finite dimensional subspace of $X$,
and an approximation scheme $(A_n)$
is defined by setting, for $n \in \N$,
$$
A_n = E + \cup_{F \subset I, |F| \leq n} \spn[X_i : i \in F] .
$$
Then the approximation scheme $(A_n)$ has Property (P), and consequently,
satisfies Shapiro's Theorem.
\end{theorem}

\begin{theorem}\label{biorth_syst}
For a complete minimal system $(x_i)_{i \in I}$ in a Banach space $X$,
consider the approximation scheme
$A_n = \{ \sum_{i \in F} \alpha_i x_i : F \subset I, \, |F| \leq n\}$
($n \geq 0$). Then for every $n$ there exists a norm $1$ $y \in X$ such that
(in the above notation) $E(y, A_{n-1}) > 1/(2n)$.
Consequently, the approximation scheme $(A_n)$ satisfies Shapiro's Theorem.
\end{theorem}



To prove Theorem~\ref{decomposition}, we need

\begin{lemma}\label{riesz}
Suppose $Y$ is a subspace of a quasi-Banach space $X$, with $\overline{Y} \subsetneq X$.
Then for every $\varepsilon > 0$ there exists
$w \in X$ such that  $\|w\| \leq 1$, and $dist(w, Y) \geq 1 - \varepsilon$.
\end{lemma}

\begin{proof}
Take $x\in X\setminus \overline{Y}$. Then $d=E(x,Y)>0$, and there exists
$y_0\in Y$ such that $d\leq \|x-y_0\|\leq \frac{1}{1-\epsilon}d$. Set
$z=x-y_0$ and $w=z/\|z\|\in S(X)$. Then
$$
\|w-y\|=\frac{1}{\|x-y_0\|}\|x-(y_0+y\|z\|)\|\geq
\frac{1}{\|x-y_0\|}E(x,Y)\geq (1-\epsilon)
$$
for any $y\in Y$.
\end{proof}

\begin{proof}[Proof of Theorem~\ref{decomposition}]
For $i \in I$, denote by $P_i : X \to X$ by setting $P_i x = x$ if $x \in X_i$,
and $P_i x = 0$ if $x \in \overline{\spn[X_j : j \in I \backslash \{i\}]}$.
Then $C_0 = \sup_i \|P_i\|$ is finite. Let $m = \dim E + 1$. We shall find
$y \in X$ such that  $\|y\| \leq 1$, and
$E(y, A_{n-1}) \geq (2 C^2 C_0 m^{1/p} n^{1/p})^{-1}$.

To this end fix disjoint subsets $S_1, \ldots, S_n \in I$,
of cardinality $m$ each. For $1 \leq k \leq n$, set
$Y_k = \spn[X_i : i \in S_k]$. Then $Q_k = \sum_{i \in S_k} P_i$
is a projection onto $Y_k$, satisfying $Q_k \spn[X_i : i \notin S_k] = 0$.
By the assumptions about $X$,
$\|Q_k\| \leq C (\sum_{i \in S_k} \|P_i\|^p)^{1/p} = C C_0 m^{1/p}$.
Moreover, for each $k$, $\dim Q_k(E) < m$, while $\dim Y_k \geq m$.
By Lemma~\ref{riesz}, there exists a norm one $y_k \in Y_k$
such that  $E(y_k, Q_k(E)) > 1/2$.

Now consider $y = (y_1 + \ldots + y_n)/(C n^{1/p})$. Clearly, $\|y\| \leq 1$.
It remains to show that, for any $e \in E$, any $F \subset I$ of
cardinality not exceeding $n-1$, any family of scalars $(\alpha_i)_{i \in F}$,
and any family $x_i \in X_i$ (once again, $i \in F$), we have
$\|y - (e + \sum_{i \in F} \alpha_i x_i)\| \geq (2 C^2 C_0 m^{1/p} n^{1/p})^{-1}$.
Find $k$ such that  $S_k \cap F = \emptyset$. Then
\begin{align*}
\|Q_k\| \|y - (e + \sum_{i \in F} \alpha_i x_i)\|
&
\geq
\|Q_k(y - (e + \sum_{i \in F} \alpha_i x_i))\|
\\
&
=
\|Q_k y - Q_k e\| \geq \frac{1}{C n^{1/p}} E(y_k, Q_k(E)) \geq \frac{1}{2 C n^{1/p}} .
\end{align*}
We complete the proof by recalling that $\|Q_k\| \leq C C_0 m^{1/p}$.
\end{proof}

The following lemma (necessary for the proof of Theorem~\ref{biorth_syst})
may be known to experts, although we couldn't find its statement anywhere.
Throughout, we use $S(X)$ and $B(X)$ to denote the unit sphere, respectively
the closed unit ball, of $X$.

\begin{lemma}\label{complemented}
Suppose $X$ is a Banach space, $E$ is a weak$^*$-closed subspace of $X^{**}$,
and $Z$ is a subspace of $X$, such that $\dim X/Z < \infty$, and
$\dim X^{**}/E > \dim X/Z$ ($E$ can be of finite or infinite codimension).
Then for every $c < 1$ there exists $x \in S(Z)$ such that $dist(x, E)_{X^{**}} \geq c$.
\end{lemma}

\begin{proof}
Suppose, for the sake contradiction, that the statement of the lemma is false.
Then there exists $c \in (0,1)$ with the property that, for every $x \in B(Z)$,
there exists $e \in E$ such that $\|x-e\|_{X^{**}} \leq c$. By the triangle inequality,
$\|e\|_{X^{**}} \leq 1+c$, hence $B(Z) \subset (1+c) B(E) + c B(X^{**})$. The set on the
right is weak$^*$ closed (even weak$^*$ compact).
Taking the weak$^*$ closure of the left hand side, we obtain
\begin{equation}
B(Z^{\perp\perp}) \subset (1+c) B(E) + c B(X^{**})
\label{within_c}
\end{equation}
Let $W = Z^{\perp\perp} \cap E$, and consider the quotient map
$q : X^{**} \to X^{**}/W$. This map takes $Z^{\perp\perp}$
and $E$ to $Z^\prime = Z^{\perp\perp}/W$ and $E^\prime = E/W$, respectively.
Then $\dim E^\prime < \infty$, and $\dim Z^\prime > \dim E^\prime$.
By the well-known result by Krasnoselskii, Krein, and Milman
(see e.g.~\cite[Lemma 1.19]{HMVZ}), there exists $z^\prime \in Z^\prime$
such that $c < dist(z^\prime, E^\prime)_{X^{**}/W} = \|z^\prime\|_{X^{**}/W} < 1$.
Find $z \in Z^{\perp\perp}$ such that $\|z\| \leq 1$, and $q(z) = z^\prime$.
For every $e \in E$, we then have
$\|z - e\|_{X^{**}} \geq \|q(z-e)\|_{X^{**}/W}  \geq dist(z^\prime, E^\prime)_{X^{**}/W} > c$,
which contradicts \eqref{within_c}.
\end{proof}

\begin{proof}[Proof of Theorem~\ref{biorth_syst}]
By Hahn-Banach Theorem, there exist linear functionals $f_i \in X^*$,
satisfying $ \langle x_i, f_j \rangle = \delta_{ij}$ for $i,j \in I$.
Throughout the proof,
we consider the functionals $f_i$ as acting on $X^{**}$,
and their kernels $\ker f_i$ as subsets of $X^{**}$.
We also identify $X$ with its canonical image in $X^{**}$.

We shall construct a sequence of finite disjoint sets $S_j \subset I$
such that for any $j$ there exists a norm $1$ $y_j \in \spn[x_i : i \in S_j]$
with the property that $E(y_j, \spn[x_i : i \notin S_j]) > 1/2$.
Once this is done, let $y = (y_1 + \ldots + y_n)/n$. Clearly $\|y\| \leq 1$.
It remains to show that $\|y - \sum_{i \in F} \alpha_i x_i\| > 1/(2n)$
for any $F \subset I$ of cardinality less than $n$. As the sets $S_j$ are
disjoint, there exists $j$ such that $S_j \cap F = \emptyset$.
Then
$$
\|y - \sum_{i \in F} \alpha_i x_i\| =
\Big\| \frac{y_j}{n} + \frac{1}{n} \sum_{k \neq j} y_k -
\sum_{i \in F} \alpha_i x_i \Big\| \geq
\frac{1}{n} E(y_j, \spn[x_i : i \notin S_j]) > \frac{1}{2n} .
$$

We construct the sets $S_j$ and vectors $y_j$ inductively. Let $S_0 = \emptyset$.
Suppose the sets $S_j$ have already been obtained for all $j \leq m-1$ ($m \in \N$).
Let us construct $S_m$ and $y_m$. Let $T = \cup_{j < m} S_j$.
Introduce the spaces $E_0 = \cap_{i \in I} \ker f_i \hookrightarrow X^{**}$,
and $E_T = \spn[x_i : i \in T] \hookrightarrow X$. 
Define the projection $Q_T$ from $X^{**}$ onto $E_T$ by setting
$Q_T x = \sum_{i \in T} \langle f_i, x \rangle x_i$.
Clearly, $E_0$ is weak$^*$ closed, and $E_T$ is weak$^*$ closed due to being
finite dimensional. As $E_0 \subset \ker Q_T$, we conclude that
$E = E_0 + E_T$ is also weak$^*$ closed. Note that the set $(f_i)$ is linearly
independent, hence $\dim X^{**}/E_0 = \infty$.

Now set $Z = X \cap (\cap_{i \in T} \ker f_i)$. As $\dim X/Z < \infty$,
Lemma~\ref{complemented} implies the existence of $z \in B(Z)$ satisfying
$dist(z, E)_{X^{**}} > 5/6$. As $\spn [x_i : i \in I]$ is dense in $X$, there
exists $z_1 \in S(\spn [x_i : i \in I])$ such that  $\|z - z_1\| < 1/(12 \|Q_T\|)$,
and $dist(z_1, E)_{X^{**}}  > 5/6$. Let $z_2 = z_1 - Q_T z_1$. Then
$$
\|z_2 - z_1\| = \|Q_T z_1\| = \|Q_T (z_1 - z)\| \leq \|Q_T\| \|z_1  - z\| < 1/12 ,
$$
hence $\|z_2\| < 13/12$, and $dist(z_2, E)_{X^{**}} > 5/6 - 1/12 = 3/4$.
Letting $y = z_2/\|z_2\|$, we conclude that $dist(y, E)_{X^{**}} > 2/3$.

By our construction, there exists a finite set $S \subset I \backslash T$
such that $y \in \spn[x_i : i \in S]$.
Let $I^\prime = I \backslash (T \cup S)$, and show that there exists
a finite set $F \subset I^\prime$ such that
$$
E(y, \spn[x_i : i \in T \cup (I^\prime \backslash F)]) > 2/3 .
$$
Once such a set is found, then we can take $y_m = y$, and $S_m = S \cup F$.

Suppose otherwise. Then, for every $F$ as above, there exists
$y_F \in \spn[x_i : i \in T \cup (I^\prime \backslash F)]$,
satisfying $\|y - y_F\| \leq 2/3$. Observe that the set ${\mathcal{F}}(I^\prime)$
of finite subsets of $I^\prime$ forms a net, ordered by inclusion. More
precisely, for $F_1, F_2 \in {\mathcal{F}}(I^\prime)$, we say
$F_1 \prec F_2$ if $F_1 \subset F_2$. For any $F_1, F_2 \in {\mathcal{F}}(I^\prime)$,
there exists $F_3 \in {\mathcal{F}}(I^\prime)$ such that  $F_1 \prec F_3$ and $F_2 \prec F_3$
(in fact, we can take $F_3 = F_1 \cup F_2$).
By the triangle inequality, $\|y_F\|_{X^{**}}= \|y_F\|\leq 5/3$ for each $F$.
As the unit ball of $X^{**}$ is weak$^*$-compact, there exits a subnet
${\mathcal{A}}$ of ${\mathcal{F}}(I^\prime)$ such that  the net
$(y_F)_{F \in {\mathcal{A}}}$ converges weak$^*$ to some $x \in X^{**}$.
Then $\|y - x\|_{X^{**}} \leq \sup_F \|y - y_F\|\leq 2/3$.
Note that, for any $j \in F \cup S$, $\langle f_j, y_F \rangle = 0$.
Moreover, for every $F \in {\mathcal{F}}(I^\prime)$, there exists
$G \in {\mathcal{A}}$ containing $F$.
Therefore, $\langle f_j, x \rangle = 0$ for any
$j \in I^\prime \cup S = I \backslash T$.
Then $\langle f_j, x - Q_T x \rangle = 0$ for any $j \in I$, hence
$x - Q_T x \in E_0$, and therefore, $x \in E$.
This, however, contradicts $dist(y, E)_{X^{**}} > 2/3$.
\end{proof}



As an application, consider a compact set $K\subset \C$, such that $\Omega=\textbf{Int}(K)$
is a Jordan domain, and $C=\partial K$ is a rectifiable Jordan curve.
Define the family of {\it Faber polynomials} $\{F_n(z)\}_{n=0}^{\infty}$,
associated with $K$. Let $\phi$ be the Riemann mapping function defined from
$\C\setminus \overline{\D}$ onto $\C\setminus K$. Then
\[
F_n(z)= \frac{1}{2\pi i}\int_{|w|=1}\frac{w^n\phi'(w)}{\phi(w)-z}dw.
\]
These polynomials play a main role in complex approximation theory, so the dictionary
$\mathcal{D}=\{F_n\}_{n=0}^{\infty}$ is of interest
(see \cite{suetin}, \cite{curtiss} for more information on Faber polynomials).

\begin{corollary}
Let K be a closed Jordan domain of bounded boundary rotation,
such that the boundary $C=\partial K$ has no external cusps. Let $\mathcal{D}=\{F_n\}_{n=0}^{\infty}$, where $F_n(z)$ denotes the $n$-th Faber polynomial associated to $K$. Then $\mathcal{D}$ satisfies Shapiro's theorem on $A(K)$.
\end{corollary}
\begin{proof}
We show that, for $K$ as in the statement of the theorem, the Faber polynomials
form a complete minimal system in $A(K)$. An application of Theorem \ref{biorth_syst}
completes the proof.

On $K=\overline{\D}$, the Faber polynomials are the monomials $e_n$ ($e_n(z) = z^n$).
It is well known that $\spn[e_n : n \geq 0]$ is dense in $A(\overline{\D})$.
Moreover, the functionals $f \mapsto \widehat{f}(n)$ are biorthogonal to the $e_n$'s.
In the general case, by \cite[Chapter 1, Section C]{gaier}, there exists a
bounded injective operator $T : A(\overline{\D}) \to A(K)$, such that
$T e_n = F_n$ for any $n \geq 0$. By \cite{anderson}, the range of $T$ coincides with $A(K)$,
and $\|T^{-1}\|<\infty$. As an isomorphic image of a complete minimal system is
again a complete minimal system, we are done.
\end{proof}


\subsection{Generalized Haar schemes}
In this section we introduce and investigate the class of generalized Haar
families in spaces of functions (numerous examples will be given below).
Suppose, for each $n$, $A_n$ is a set of continuous functions on $\Omega$.
We say that the family $\{A_n\}$ is {\it generalized Haar}
if there exists a function $\psi = \psi_{\{A_n\}} : \N \to \N$ such that
no non-zero function of the form $\Re g$ ($g \in A_n$) has more than
$\psi(n) - 1$ zeroes on $\Omega$. Finally, the approximation scheme $(X,\{A_n\})$ is named ``generalized Haar''
if  $\{A_n\}$ is a generalized Haar system.

Very often, we consider the approximation schemes arising from dictionaries
(see \eqref{nterm} for the definition). We say that a dictionary ${\mathcal{D}}$ is a
{\it generalized Haar system} if the  family $\{\Sigma_n(\mathcal{D})\}$ is Haar.

In the four examples below, we exhibit some generalized Haar dictionaries.
The space $X$ is either $C([a,b])$, or $L_p(a,b)$ ($0 < p < \infty$), and
$\psi(n) = n$.
\begin{enumerate}
\item The dictionary ${\mathcal{D}}$, consisting of the functions
$f_\lambda(t) = t^\lambda$ ($\lambda \in \R$) on an interval $[a,b]$ with $0 < a < b$.
Indeed, these functions form a generalized
Haar system \cite[Section 3.1]{borwein}.
As polynomials are dense in $C([a,b])$, $\spn[{\mathcal{D}}]$ is dense in $X$.
\item The dictionary ${\mathcal{D}}$, consisting of functions $f_k(t) = t^k$
($k \in \N \cup \{0\}$) on arbitrary $[a,b]$. Indeed, the family $(f_k)$
forms a generalized Haar system on subintervals of $(0,\infty)$,
and of $(-\infty,0)$.
\item The dictionary ${\mathcal{D}}$, consisting of functions
$f_\lambda(t) = \exp(\lambda t)$ ($\lambda \in \R$), with arbitrary $[a,b]$.
In this case, the density of $\spn[{\mathcal{D}}]$ in $C([a,b])$ can be deduced,
for instance, from Stone-Weierstrass Theorem.
Furthermore, ${\mathcal{D}}$ is a generalized Haar system, by
\cite[Chapter 3]{borwein}.
\item The dictionary ${\mathcal{D}}$, consisting of functions $f_k(t) = t^k$
on $\R$. Consider a weight $W$ -- that is, an $L_1$ function $W : \R \to [0,1]$.
Consider the measure $\mu$, defined by $\mu(E) = \int_E W(x) \,dx$.
Take $X$ to be either $L_p(\mu)$ ($1 \leq p < \infty$), or a set of
continuous functions $f$ on $\R$ satisfying $\lim_{t \to \infty} f(t) W(t) = 0$.
For certain weights $W$, $\spn[{\mathcal{D}}]$ is known to be dense in $X$.
For instance, this is true for $W(x) = \exp(- |x|^\alpha)$, for any $\alpha \geq 1$.
See \cite{lubinsky} for this and other results on the density of polynomials in
the weighted spaces $X$.
\end{enumerate}
Moreover, the sets of trigonometric functions $$\mathcal{T}_n=\spn[\{1,\cos(t),\sin(t),\cdots,\cos(nt), \sin(nt)\}]$$ define a Haar system on $[0,2\pi)$.
A somewhat more complicated example  of generalized Haar system involves rational functions.
For $\Omega \subset \C$, denote by $R_n(\Omega)$ the set
of all rational functions $p(z)/q(z)$, where the polynomials $p(z)=\sum_{k=0}^na_kz^k$ and $q(z)=\sum_{k=0}^nb_kz^k$
have complex coefficients and degree $\leq n$, such that $q(z)$ doesn't vanish in $\Omega$.
We also consider the set $E_n(\Omega)$ of trigonometric rational functions of degree less than $n$,
consisting of functions $t \mapsto p(e^{it})/q(e^{it})$,
where $p(z)=\sum_{k=-n}^na_kz^k$ and $q(z)=\sum_{k=-n}^nb_kz^k$,
and $q(z)\neq 0$ for all  $z\in\Omega$.


\begin{proposition}\label{haar_schemes}
If $\Omega \subset \R$, then $\{R_n(\Omega)\}$ is a generalized Haar system. Moreover, if $\Omega\subseteq \partial\D=\T$ then
$\{E_n(\Omega)\}$ is a generalized Haar family.
\end{proposition}
\begin{proof}
We handle $\{R_n(\Omega)\}$ first. If $g = p/q \in R_n(\Omega)$, then
$p=\Re p+(\Im p)i$, $q=\Re q+(\Im q)i$, and $\Re p, \Re q, \Im p, \Im q$
are polynomials of degree $\leq n$. Hence
$\Re g = \Re \left(\frac{p \overline{q}}{|q|^2}\right)=\frac{\Re p \Re q + \Im p \Im q }{|q|^2}$. As
$t \mapsto  \Re p(t) \Re q(t) + \Im p(t) \Im q(t)$ is a polynomial of
degree not exceeding $2n$, $\Re g$ must vanish if it has more than $2n$ zeroes.

Now consider $g = p/q \in E_n(\Omega)$, with $\Omega\subseteq \partial\D$. Then $p(t) = \sum_{|j| \leq n} a_j e^{itj}$, $q(t) = \sum_{|j| \leq n} b_j e^{itj}$, and
\begin{equation}\label{algebraic_rational_equals_trigonom}
g(t)=\frac{\sum_{k=-n}^n a_ke^{ikt}}{\sum_{k=-n}^n b_ke^{ikt}}=\frac{\sum_{k=0}^{2n} a_{k-n}z^k}{\sum_{k=0}^{2n} b_{k-n}z^k}; \ \ (z=e^{it}),
\end{equation}
so that
\[
\Re g(t) = \Re\left(\frac{\left(\sum_{k=0}^{2n} a_{k-n}z^k\right)\left(\sum_{k=0}^{2n} \overline{b_{k-n}}z^{-k}\right)}{\left|\sum_{k=0}^{2n} b_{k-n}z^k\right|^2}\right) = \Re\left(\frac{\sum_{k=-2n}^{2n} c_{k}z^k}{\left|\sum_{k=0}^{2n} b_{k-n}z^k\right|^2}\right) \ \ (z=e^{it}) ,
\]
Representing $c_k = \alpha_k + i \beta_k$ ($\alpha_k, \beta_k \in \R$), we see that the
zeroes of $\Re g(t)$ are the zeroes of
\begin{eqnarray*}
x(t)&=& \Re\left(\sum_{k=-2n}^{2n} (\alpha_{k}+i\beta_k)e^{ikt}\right);  \\
&=& \Re\left(\sum_{k=-2n}^{2n} (\alpha_{k}+i\beta_k)(\cos(kt)+i\sin(kt))\right)\\
&=& \sum_{k=-2n}^{2n} (\alpha_{k}\cos(kt)-\beta_k\sin(kt))\\
&=&  \alpha_{0}+\sum_{k=1}^{2n} ((\alpha_{k}+\alpha_{-k})\cos(kt)+(\beta_{-k}-\beta_k)\sin(kt)) . 
\end{eqnarray*}
Thus, $x \in \mathcal{T}_{2n}$, and we are done,
since $\{\mathcal{T}_n\}_{n=1}^{\infty}$ is a Haar family.
\end{proof}

The next result shows that ``many'' generalized Haar approximation schemes
satisfy Shapiro's Theorem.
Below, $C_0(I)$ denotes the closure of continuous functions with
compact support in the $\|\cdot\|_\infty$ norm. In particular,
$C(I) = C_0(I)$ if $I$ is a compact set.

\begin{theorem}\label{cheb_syst}
Suppose $I$ is either a finite or infinite interval in $\R$, or the unit circle $\T$.
Suppose, furthermore, that $\mu$ is a finite atomless Radon measure on $I$, and $X$ is
a quasi-Banach space of functions on $I$, satisfying $L_p(\mu) \supseteq X \supseteq C_0(I)$
with some $p > 0$. Then any generalized Haar  approximation  scheme $(X, \{A_n\})$
satisfies Shapiro's Theorem.
\end{theorem}

\begin{proof}
Without loss of generality, we can assume $\mu(I) = 1$, and
$\int |f|^p \, d\mu \leq \|f\|_X^p$ for any $f \in X$. By Closed Graph Theorem,
there exists a constant $C$ such that, for any $f \in C_0(I)$, we have
$\|f\|_X \leq C \|f\|_\infty$.

For every $n \in \N$, we find a continuous function $h : I \to [-1,1]$ with compact
support, such that $\|h\|_\infty = 1$, and $\int |h-g|^p \, d\mu > 1/5$ for any $g\in A_n$. Indeed, if such an $h$ exists, then
$\|C^{-1}h\|_X\leq 1$ and $E(C^{-1} h, A_n)_{X} > 1/(5C)^{1/p}$.
By Corollary~\ref{coro}, $(A_n)$ satisfies Shapiro's Theorem in $X$.


As the measure $\mu$ is Radon,
$1 = \mu(I) = \sup \{ \mu([\alpha, \beta]) : [\alpha, \beta] \subset I\}$.
Pick $\alpha < \beta$ in $I$ such that $A = \mu([\alpha, \beta]) > 4/5$.
Let $N = \psi(n) + 1$. 
Set $t_0 = \alpha$, $t_{4N} = \beta$. As the map $s \mapsto \mu((a,s))$
is continuous, we can find $t_0 < t_1 < \ldots < t_{4N}$ such that, for
$1 \leq j \leq 4n$, $\mu((t_{j-1},t_j)) = A/(4N) > 1/(5N)$.
Recall that, for any $a<b$, $\mu((a,b))$ is the supremum of $\int \rho \, d\mu$,
taken over all non-negative continuous functions $\rho$, supported on
$(a,b)$, and such that $\|\rho\|_\infty \leq 1$. So, for $1 \leq j \leq 4N$,
we can find continuous $h_j : \R \to [0,1]$, supported on $(t_{j-1}, t_j)$,
such that $\int_{t_{j-1}}^{t_j} h_j \, d\mu > 1/(5N)$.

We shall show that $h = \sum_{j=1}^{4N} (-1)^j h_j^{1/p}$ satisfies
$\int |h-g|^p \, d\mu > 1/5$ for any $g \in A_n$.
As the function $h$ defined above is real-valued, it suffices to
prove the inequality $\int |h - \Re g|^p \, d\mu > 1/5$.
If $\Re g$ is identically $0$, the desired inequality follows from the
definition of $h$. Otherwise, denote by ${\mathcal{S}}$ the set of points
where $\Re g$ changes sign.
As $|{\mathcal{S}}| < N$, the set ${\mathcal{F}} = \{1 \leq k \leq 2N :
(t_{2k-2}, t_{2k}) \cap {\mathcal{S}} = \emptyset\}$ has the cardinality
larger than $N$. Note that, for $k \in {\mathcal{F}}$,
$\int_{t_{2k-2}}^{t_{2k}} |h - \Re g|^p \, d \mu > 1/(5N)$.
Indeed, if $g \leq 0$ on $(t_{2k-2}, t_{2k})$, then
$$
\int_{t_{2k-2}}^{t_{2k}} |h - g|^p \, d \mu \geq
\int_{t_{2k-1}}^{t_{2k}} |h - \Re g|^p \, d \mu \geq
\int_{t_{2k-1}}^{t_{2k}} h_{2k} \, d \mu > \frac{1}{5N} .
$$
The case of $g \geq 0$ is handled similarly.
Thus,
$$
\int |h-g|^p \, d\mu \geq
\sum_{k \in {\mathcal{F}}} \int_{t_{2k-2}}^{t_{2k}} |h - g|^p \, d \mu >
|{\mathcal{F}}| \cdot \frac{1}{5N} > \frac{1}{5} ,
$$
completing the proof.
\end{proof}

A similar result holds in the analytic case. Below, $A$ and $H_p$ refer
to the disk algebra and to the Hardy space, respectively.

\begin{proposition}\label{analytic}
Suppose $X$ is a quasi-normed space of analytic functions on $\T$,
such that $A \subset X \subset H_p$ for some $p > 0$.
Suppose $(X, \{A_n\})$ is a generalized Haar approximation scheme.
Then $(X, \{A_n\})$ satisfies Shapiro's Theorem.
\end{proposition}

\begin{proof}[Sketch of the proof]
We identify $\T$ with $[0,1]$. Let $N = \psi(n)$, and consider
$h(t) = -i \exp(4\pi N i t)$. As in the proof of the previous theorem,
it suffices to show that, for any $g \in A_n$,
$\int_0^1 |\Re h(t) - \Re g(t)|^p \, dt > c/4$, where
$c = \int_0^1 |\sin u|^p \, du$. Note that $\Re h(t) = \sin (4 \pi N t)$.
For $0 \leq j \leq 4N$, set $t_j = j/(4N)$. Then
$\int_{t_{j-1}}^{t_j} |\Re h|^p = c/(4N)$ for any $j$.
If $\Re g$ is identically $0$, then
$$
\int_0^1 |\Re h - \Re g|^p \, dt =
\sum_{j=1}^{4N} \int_{t_{j-1}}^{t_j} |\Re h|^p = c .
$$
Otherwise, denote by ${\mathcal{F}}$ the set of all $k \in \{1, \ldots, 2N\}$
such that $\Re g$ doesn't vanish on $(t_{2k-2}, t_{2k})$. As $|{\mathcal{F}}| < N$,
We complete the proof as in Theorem~\ref{cheb_syst}.
\end{proof}

Another interesting Banach space is
$CBV_0(a,b)=\{f\in C([a,b]) :f(a)=0, V_{[a,b]}(f)<\infty\}$,
equipped with the norm $\|f\|_{BV}=V_{[a,b]}(f)$
(here $V_{[a,b]}(f)$ denotes the total variation of $f$).

\begin{theorem}\label{cheb_syst_BV}
Let ${\mathcal{D}}$ be a dictionary on $CBV_0(a,b)$. Suppose
${\mathcal{D}} \subset C^1([a,b])$, and
$\mathcal{D}'=\{g':g\in \mathcal{D}\}$ is a generalized Haar system on $[a,b]$.
Then ${\mathcal{D}}$ satisfies Shapiro's Theorem in $CBV_0(a,b)$.
\end{theorem}

\begin{proof}
We work with the case of $[a,b]=[0,2\pi]$. By Corollary \ref{coro},
we only need to prove that
\[
\sup_{\|f\|_{BV}=1}E(f,\Sigma_n(\mathcal{D}))\geq \frac{1}{3}\ \text{ for } n=0,1,2,\cdots.
\]
Let $N = 6 \psi(n)$, and consider $f(t) = (1 - \cos Nt)/(4N)$.
Then $\|f\|_{BV} = 4^{-1} \int _0^{2\pi} |\sin Nt| \, dt = 1$.
We show that, for any $g \in \Sigma_n({\mathcal{D}})$,
\[
\|f-g\|_{BV} \geq \int_{0}^{2\pi}|f^\prime(t) - \Re g'(t)|dt \geq \frac{1}{3} .
\]
For such a $g$, define ${\mathcal{F}}$ as the set of all $\ell \in \{1, \ldots, N\}$
with the property that $\Re g^\prime$ does not change sign on
$(2\pi(\ell-1)/N, 2\pi\ell/N)$.
Note that $|{\mathcal{F}}| \geq N - \psi(n) = 5N/6$.
Furthermore, $f^\prime$ is positive on $(2\pi(\ell-1)/N, \pi(2\ell-1)/N)$, and
negative on $(\pi(2\ell-1)/N, 2\pi\ell/N)$. One one of these two intervals,
$|f^\prime| \geq |f^\prime - \Re g|$. Furthermore,
$$
\int_{2\pi(\ell-1)/N}^{\pi(2\ell-1)/N} |f^\prime| \, dt =
\int_{\pi(2\ell-1)/N}^{2\pi\ell/N} |f^\prime| \, dt = \frac{1}{2N} .
$$
Thus, for $\ell \in {\mathcal{F}}$,
$$
\int_{2\pi(\ell-1)/N}^{2\pi\ell/N} |f^\prime - \Re g^\prime| \, dt \geq \frac{1}{2N} ,
$$
and therefore,
\[
\|f-g\|_{BV} \geq \int_{0}^{2\pi}|f^\prime(t) - \Re g'(t)|dt \geq
\sum_{\ell \in {\mathcal{F}}}
\int_{2\pi(\ell-1)/N}^{2\pi\ell/N} |f^\prime - \Re g^\prime| \, dt \geq
\frac{1}{2N} \cdot \frac{5N}{6} > \frac{1}{3} .
\]
\end{proof}


\subsection{Approximation by rational functions}\label{section_rational}
The problem of describing the possible sequences of best rational approximations for a given
function dates back at least to E.~Dolzhenko \cite{dol}.
Certain Bernstein-type results have been obtained for
approximations in the uniform norm. For instance,
if $\varepsilon_1 > \varepsilon_2 > \ldots$ and $\lim \varepsilon_m = 0$, then
there exists $f \in C(\mathbb{T})$ such that
$E(f, E_m(\mathbb{T}))_{C(\mathbb{T})} = \varepsilon_m$ for every $m$ \cite{starovoitov}
(see also \cite{naz, pek94} for related results).
Evidence suggests that the condition that the sequence $\{\varepsilon_m\}$
is strictly increasing can be weakened. By \cite{starovoitov03}, for every sequence
$\{\varepsilon_m\} \searrow 0$ there exists $f \in C[0,1]$ such that
$E(f, R_{2^m-1}([0,1]))_{C[0,1]} = \varepsilon_m$ for every $m$.
On the other hand, Bernstein's Lethargy Theorem cannot be perfectly replicated for
rational approximation in $L_p$: by \cite{Lev}, for any $f \in L_p(0,1)$, the sequence
$E(f, R_m([0,1]))_{L_p}$ is either strictly decreasing, or eventually null.

This section attempts to (partially) answer Dolzhenko's question by  proving Shapiro's Theorem
for rational approximations in a variety of function spaces.

\begin{theorem}\label{rationals}
Take $0 < p < \infty$.
The following approximation schemes satisfy Shapiro's Theorem:
\begin{enumerate}
\item $(X, \{R_n(I)\})$, where $I$ is a real interval and $C_0(I) \subseteq X \subseteq L_p(I)$, and $\overline{C_0(I)}^X=X$.
\item $(X, \{E_n(I)\})$, where is a real interval and $C_0(I) \subseteq X \subseteq L_p(I)$, and $\overline{C_0(I)}^X=X$.
\item $(X, \{E_n(\T)\})$, where $C(\T) \subseteq X \subseteq L_p(\T)$, and $\overline{C(\T)}^X=X$.
\item
$(X, \{R_n(\partial \D)\})$, where $C(\partial \D) \subseteq X\subseteq L_p(\partial \D)$, and $\overline{C(\partial \D)}^X=X$.
\item
$(X, \{R_n(\overline{\D})\})$, where $A\subseteq X\subseteq H_p$, and $\overline{A}^X=X$.
\end{enumerate}
\end{theorem}

\begin{proof} We start noting that the densities $\overline{C_0(I)}^X=X$, $\overline{C(\T)}^X=X$ and $\overline{A}^X=X$ are assumed to guarantee that, if our approximation scheme $(A_n)$ is dense, for example, in $C_0(I)$ then it is also dense in $X$. To prove this, take $x\in X$, $\varepsilon>0$ arbitrarily small and $C>0$ such that $\|\cdot\|_X\leq C\|\cdot\|_{\infty}$. Look for $f\in C_0(I)$ such that $\|x-f\|^q_X<\frac{\varepsilon}{2}$ and $a\in\bigcup_nA_n$ such that $\|f-a\|^q_{\infty}\leq \frac{\varepsilon}{2C^q}$. Then $\|x-a\|_X^q\leq \|x-f\|_X^q+\|f-a\|_X^q\leq \frac{\varepsilon}{2}+C^q\|f-a\|_{\infty}^q\leq \varepsilon$.

Now (1), (2) and (3) are direct consequences of
Proposition~\ref{haar_schemes} and Theorem~\ref{cheb_syst}.

To deduce (4) from (3),
consider a map $U$, taking a function $f : \partial \D \to \C$
to $\tilde{f} : \T \to \C$, where $\tilde{f}(t) = f(e^{it})$.
Clearly, $U$ is an isometry from $C(\partial \D)$ onto $C(\T)$,
and from $L_p(\partial \D)$ onto $L_p(\T)$. Hence it is clear that $U$ maps the space $X$ isometrically onto a space $Y$ which
satisfies $C(\T)\subseteq Y\subseteq L_p(\T)$.  Moreover, the equality \eqref{algebraic_rational_equals_trigonom} implies that
$U$ maps $R_{2n}(\partial \D)$ onto $E_n(\T)$.

To establish (5), note that the elements of $A$ or $H_p$ are uniquely
determined by their restrictions to $\partial \D$
(see e.g.~Appendix 3 of \cite{LGM}).
Thus, we identify our functions on $\D$ with functions
on $\partial \D$. The density of $\cup_n R_n(\overline{\D})$ in $X$
follows from the proof of Theorem 1.5.2 of \cite{LGM} and the density of $A$ in $X$.
Identifying $\partial \D$ with $\T$, we complete the proof by applying Proposition~\ref{analytic}.
\end{proof}

\begin{corollary}\label{weighted}
Suppose $X$ is either $C(\overline{\R})$ (the set of continuous functions
$f$ on $\R$ for which $\lim_{t \to +\infty} f(t)$ and $\lim_{t \to -\infty} f(t)$ exist
and are equal), or $L_p(W, \R)$, where $0 < p < \infty$, and the weight $W$
is given by $W(x) = 2/(1+x^2)$.
For $n \in \N$, denote by $R_n(\overline{\R})$ the set of rational functions
$p/q$, where $\deg p \leq \deg q < n$, and $q$ has no real roots.
Then the approximation scheme $(X, \{R_n(\overline{\R})\})$
satisfies Shapiro's Theorem.
\end{corollary}

\begin{proof}
In this proof, we use some ideas of \cite[Section 1.5]{LGM}.
As before, identify $\T$ with $[-\pi, \pi]$. Consider the map
$\Phi : \T \to \R : t \mapsto \tan(t/2)$ ($-\pi \sim \pi$ is taken to $\infty$).
The map $U_\Phi : f \mapsto f \circ \Phi$ is then an isometry from $Y$ onto $X$,
where $Y$ is either $C(\T)$ or $L_p(\T)$.

Denote by $R_n^\prime(\overline{\R})$ the set of all functions
$p/q \in R_n(\overline{\R})$ for which all the roots of $q$ are distinct.
Similarly, let $E_n^\prime(\T)$ the set of all functions
$p/q \in E_n(\T)$ for which all the roots of $q$ are distinct.
A small perturbation argument shows that $R_n^\prime(\overline{\R})$
($E_n^\prime(\T)$) is dense in $R_n(\overline{\R})$ (resp.~$E_n(\T)$).

Any $f \in R_n^\prime(\overline{\R})$ can be written as
$f = \alpha_0 \one + \sum_{j=1}^m \alpha_j g_{c_j}$, with $m < n$.
Here, $\one(x) = 1$, and $g_c(x) = (1-ix)/(x-c)$ ($c \notin \R$).
By formula (5.13) of \cite{LGM}, $g_c \circ \Phi = \alpha f_z$, where
$z = (i-c)/(i+c)$, $\alpha$ is a numerical constant, depending on $z$,
and $f_z(t) = 1/(e^{it} - z)$. Thus, $\Phi$ implements a
$1-1$ correspondence between $R_n^\prime(\overline{\R})$ and $E_n^\prime(\T)$.

It is established in \cite[Section 1.5]{LGM} that $(Y, \{E_n^\prime(\T)\})$
is an approximation scheme. By Proposition~\ref{rationals},
$(X, \{E_n^\prime(\T)\})$ satisfies Shapiro's Theorem.
As $U_\Phi$ is an isometry, $(X, \{R_n^\prime(\overline{\R})\})$ is also
an approximation scheme, satisfying Shapiro's Theorem.
The density of $R_n^\prime(\overline{\R})$ in $R_n(\overline{\R})$
completes the proof.
\end{proof}

\begin{remark}\label{other_ideas}
%
Below we outline some alternative approaches to the results of
Theorem~\ref{rationals}.
For instance, one can show that $(C([a,b]), \{R_n([a,b])\})$ satisfies
Shapiro's Theorem, one can use a Bernstein-type inequality due to Dolzhenko \cite{D}:
the total variation of $f\in R_n([a,b])$ satisfies
$V_{[a,b]}(f)=\int_a^b|f'(t)|dt \leq 2n\|f\|_{C([a,b])}$.
The space of continuous functions of bounded variation
$Y=CBV[a,b]$, equipped with the norm
$\|f\|_{Y}=\|f\|_{C[a,b]}+V_{[a,b]}(f)<\infty$, is a proper
dense linear subspace of $C([a,b])$. An application of Bernstein's
Inequality (Theorem~\ref{central}) to $\| \cdot \|_Y$ completes the proof.

For $L_p(a,b)$, $(1<p<\infty)$, we may use a result by Pekarskii
\cite{P} (see also Theorem 1.1 in \cite[page 300]{LGM}): for $k=1,2,\cdots$,
$1<p<\infty$, and $\gamma=(r+\frac{1}{p})^{-1}$, the approximation scheme
$(R_n([-1,1]))$ satisfies a Bernstein-style inequality:
\[
\|f^{(k)}\|_{L_{\gamma}(-1,1)}\leq C(p,k)n^k\|f\|_{L_p(-1,1)} \ \ (f\in R_n([-1,1])), \ n=1,2,\cdots.
\]
Hence we can use Theorem \ref{central} for $(L_p(-1,1),\{R_n([-1,1])\})$ with
$Y=\{f\in L_p(-1,1): f^{(k)}\in L_{\gamma}(-1,1)\}$ (with the norm $\|f\|=\|f\|_p+\|f^{(k)}\|_{L_{\gamma}(-1,1)}$).

One can also tackle $L_p$ by using the strong relation between rational approximation
and approximation by spline functions with free knots.
By Theorem 6.8 from \cite[page 340]{LGM}, for $1<p< \infty$, $0<q\leq\infty$ and $0<\alpha<r$, the approximation spaces
$\mathbb{R}^{\alpha}_{p,q}=\{f\in L_p(0,1):\{n^{\alpha-\frac{1}{q}}E(f,R_n(0,1))_{L_p}\}\in \ell_q\}$ and  $\mathbb{S}^{\alpha}_{p,q}=\{f\in L_p(0,1):\{n^{\alpha-\frac{1}{q}}E(f,S_{n,r}(0,1))_{L_p}\}\in \ell_q\}$ are the same
(with equivalent norms). Theorem \ref{new_splines} below guarantees that  $\mathbb{R}^{1}_{p,\infty}=\mathbb{S}^{1}_{p,\infty}$ is a strict subset of $L_p(0,1)$.
Hence there exists a function $f\in L_p(0,1)$ such that $E(f,R_n(0,1))_{L_p}\neq \mathbf{O}(n^{-1})$ and the result follows from Corollary \ref{coro2}.

We next sketch an argument showing that
$(C(\partial\D),\{R_{n}(\partial\D)\}_{n=0}^\infty)$ satisfies Shapiro's Theorem.
That is, by Corollary~\ref{coro}, we have to find $f \in C(\partial\D)$ such that
the sequence $E(f, R_{n}(\partial\D)) \geq \alpha_n$, for a prescribed sequence
$\{\alpha_n\} \searrow 0$. Having already shown that
$(C([0,1]), \{R_{n}([0,1])\})$ satisfies Shapiro's Theorem,
we conclude that there exists $h \in C([0,1])$ such that
$E(h, R_{n}([0,1])) \geq \alpha_n$ for every $n$.
Extend $h$ to a bounded continuous function $g$ on $\R$, for which
$\lim_{t\to +\infty}g(t)= \lim_{t\to -\infty}g(t)$ exist and are equal.
Clearly, $E(g, R_{n}(\R)) \geq E(h, R_{n}([0,1])) \geq \alpha_n$ for every $n$.
Finally, consider the linear fractional transformation
$w(t)=\frac{(i-1)t+(i+1)}{(1-i)t+(i+1)}$, mapping $\R$ onto $\partial\D$.
Define $f = g\circ w^{-1} \in C(\partial\D)$.
Note that $R \in R_n(\R)$ if and only if $R\circ w^{-1}\in R_n(\partial\D)$. Therefore,
\[
E(f,R_n(\T))_{C(\partial\D)}=E(f\circ w,R_n(\R))_{C(\R)} \geq E(h,R_n([0,1]))_{C[0,1]}
\geq \alpha_n
\]
for every $n$. We conclude that $(C(\partial\D),\{R_{n}(\partial\D)\})$ satisfies Shapiro's
Theorem. To deduce from this that $(C(\T),\{E_{n}(\T)\}_{n=0}^\infty)$ satisfies Shapiro's
Theorem, observe that \eqref{algebraic_rational_equals_trigonom} guarantees that
$E(f,E_n(\T))_{C(\T)}=E(f,R_{2n}(\partial\D))_{C(\partial\D)}$.
\end{remark}

\subsection{Approximation by splines}\label{further_examples}
In this subsection, we show that some ``very redundant'' approximation
systems based on splines satisfy Shapiro's Theorem.
Let us denote by $\mathcal{S}_{n,r}(I)$ the set
of polynomial splines of degree less than $r$ with $n$ free knots (nodes) on the interval $I$.
For any pair of sequences $0 \leq r_1 \leq r_2 \leq \ldots$ and $1\leq n_1<n_2< \cdots $
the sets $A_i = \mathcal{S}_{n_i,r_i}([a,b])$ form an approximation scheme in
$C[a,b]$ or $L_p(a,b)$, with $1 \leq p < \infty$ (in the case of $C[a,b]$,
we assume that the splines in question are continuous).

\begin{theorem}\label{new_splines}
The approximation scheme defined above (either in $C([a,b])$, or in $L_p(a,b)$, for
$0 < p < \infty$) satisfies Shapiro's Theorem.
\end{theorem}


\begin{proof}
The case of $C([a,b])$ follows from \cite[Theorem 3.1]{almiradeltoro2}. When working with $L_p(a,b)$ ($0<p<\infty$), assume with no loss of generality
that $[a,b]=[0,1]$. For a fixed $r \in \N$,
consider the approximation scheme
$(L_p(0,1), \{B_{n,r}\}_{n=1}^{\infty})$, where $B_{n,r} = \mathcal{S}_{n,r}([0,1])$.
Pick $0 < \alpha < \min\{r,1/p\}$, and find $t > 0$
satisfying $1/t = \alpha + 1/p$.
By Theorem 8.2 of \cite[page 386]{devore}, a Bernstein's inequality holds:
\begin{equation}\label{petrushev_inequality}
\|f\|_{B_{t}^{\alpha}(L_{t}(0,1))}\leq
Cn^{\alpha}\|f\|_{L_p(0,1)}\; (f\in B_{n,r}) , n=1,2,\cdots .
\end{equation}
Here, $B_{q}^\alpha(L_p(\Omega))$ denotes the classical Besov space on $[0,1]$
(defined using the modulus of smoothness $w_r(f,t)_p$).
By \cite[Corollary 3.1]{HS}, $B_t^\alpha(L_t(0,1))$ embeds into  the classical Lorentz space
$L_{p,t}(0,1)$. Furthermore, as $t < p$, $L_{p,t}$ embeds into $L_{p,p} = L_p$
(see e.g.~\cite[Theorem 1.9.9]{brundykrugljak}).
It is easy to show that the last embedding is proper.
By Theorem~\ref{central}, the approximation scheme
$(L_p(0,1), \{B_{n,r}\}_{n=1}^{\infty})$ satisfies Shapiro's Theorem.
We complete the proof by applying Corollary \ref{increasing}.
\end{proof}

\subsection{$n$-term approximation}
In this section we study Shapiro's Theorem for $n$-term approximation.
More precisely, suppose $\mathcal{D}$ is a dictionary, and
$(\Sigma_n({\mathcal{D}})$ is the associated approximation scheme
(defined in \eqref{nterm}). Then
$\Sigma_n(\mathcal{D})+\Sigma_n(\mathcal{D})= \Sigma_{2n}(\mathcal{D})$, so that Theorem \ref{condicion} is applicable in this context.  Obviously the properties of the sequence of errors
$E(x,\Sigma_{n}(\mathcal{D}))$ strongly depend on the dictionary $\mathcal{D}$.
For example, if $\overline{\mathcal{D}}^{X}=X$, then
$E(x,\Sigma_{n}(\mathcal{D})) = 0$ for all $n\geq 1$ and the dictionary is
``too rich'' to be of interest.

For the sake of brevity, we say that a dictionary ${\mathcal{D}}$
{\it satisfies Shapiro's Theorem} in a quasi-Banach space $X$ if
the approximation scheme
$(X, \{\Sigma_n({\mathcal{D}})\})$ satisfies Shapiro's Theorem.

Proposition~\ref{stability} implies that the dictionaries satisfying Shapiro's
Theorem are stable under small perturbations:

\begin{corollary}\label{dict_stable}
Suppose a quasi-Banach space $X$ is such that there exists
$p \in (0,1]$, for which any $x_1, x_2 \in X$ satisfy
$\|x_1 + x_2\|^p \leq \|x_1\|^p + \|x_2\|^p$.
Consider the dictionaries ${\mathcal{D}}_1=\{u_{i}\}_{i\in I}$ and ${\mathcal{D}}_2=\{e_{i}\}_{i\in I}$ in $X$, such that ${\mathcal{D}}_1$
satisfies Shapiro's Theorem. Suppose, furthermore, that there exists
$\lambda \in (0,1)$ such that $\|\sum_i a_i (u_i-e_i)\|\leq \lambda\|\sum_i a_i e_i\|$
for any family $(a_i)_{i \in I}$ with finitely many non-zero entries.
Then ${\mathcal{D}}_2$ satisfies Shapiro's Theorem.
In particular, ${\mathcal{D}}_2$ satisfies Shapiro's Theorem
in the following two situations:
\begin{enumerate}
\item
$(\sum |a_i|^p)^{1/p}\leq c \|\sum a_i e_i\|$
for arbitrary scalars $a_i$, and $\sup_{i\in I}\|u_i-e_i\|< c^{-1}$.
\item
$\sup |a_i|\leq c \|\sum a_i e_i\|$
for arbitrary scalars $a_i$, and $(\sum_{i\in I}\|u_i-e_i\|^p)^{1/p} < c^{-1}$.
\end{enumerate}
\end{corollary}

Note that the inequality $\sup |a_i|\leq c \|\sum a_i e_i\|$ (with an appropriate
constant $c$) is satisfied if $(e_i)$, or even if $(e_i)$ arises from a bounded biorthogonal
system.


Below we give several examples of redundant dictionaries satisfying
Shapiro's Theorem.

\begin{proposition} \label{diccionarioredundante}
Let $\mathcal{D}=\{\chi_{(a,b)}:0\leq a<b\leq 1\}$ be the set of
characteristic functions of subintervals of $[0,1]$.
Then $(L_p(0,1),\{\Sigma_n(\mathcal{D})\}_{n=0}^{\infty})$ satisfies Shapiro's theorem
for $0 < p < \infty$.
\end{proposition}

\begin{remark}\label{binary}
Consider the dictionary ${\mathcal{D}}^\prime \subset {\mathcal{D}}$, consisting of characteristic functions of binary intervals in $L_p(0,1)$. By \cite{liv}, the greedy algorithm in this setting
converges ``very fast,'' when $f \in L_p(0,1)$ is such that the sequence
$(E(f, \Sigma_n({\mathcal{D}}^\prime)))$ decreases in a certain controlled manner.
The result above shows that, in general, $(E(f, \Sigma_n({\mathcal{D}}^\prime)))$
may decrease arbitrarily slowly.
\end{remark}

\begin{proof} 
It is not difficult to prove, by induction on $n$, that any
element of $\Sigma_n(\mathcal{D})$ can be written as a linear combination of
at most $2n+1$ characteristic functions of intervals with non-empty interiors.
This, in turn, implies $\Sigma_n(\mathcal{D})\subseteq \mathcal{S}_{4n+2,1}(0,1)$,
and the result follows from Theorem \ref{new_splines}.
\end{proof} 

Shapiro's Theorem also holds for the dictionary of imaginary exponentials
${\mathcal{D}}=\{t \mapsto \exp(i \lambda t): \lambda \in \R\}$  on any interval $[a,b]$.
Indeed, the theorem below deals with ridge functions, and includes these exponentials
as a particular case (see e.g.~\cite{cheney_light} for an introduction to ridge functions).
Suppose $\Pi = \prod_{i=1}^N [A_i, B_i]$ is a parallelepiped in $\R^N$,
and the dictionary ${\mathcal{D}}$ consists of functions
$f(t) = \exp(-i \langle \alpha, t \rangle)$, with $\alpha \in \R^N$ and $t \in \Pi$
($\langle \alpha, t \rangle =\sum_{i=1}^N\alpha_it_i$ denotes the usual scalar
product of $\R^N$).

\begin{theorem}\label{ridge}
Suppose $\Pi$ is parallelepiped, and
$X$ is a Banach space
of functions on $\Pi$ such that $L_1(\Pi) \supset X \supset C(\Pi)$, and that
$\overline{\spn[{\mathcal{D}}]} = X$. Then ${\mathcal{D}}$ satisfies Shapiro's Theorem.
\end{theorem}

The spaces $X$ with the properties described above include $L_p(\Pi)$
($1 \leq p < \infty$) and $C(\Pi)$. Indeed, $\spn[{\mathcal{D}}]$ is closed under
multiplication, and separates points in $\Pi$. By Stone-Weierstrass Theorem,
$\spn[{\mathcal{D}}]$ is dense in $C(\Pi)$. Furthermore, $C(\Pi)$ is
dense in $L_p(\Pi)$ for $1 \leq p < \infty$.


\begin{proof}
By scaling, we can assume $\Pi = [0,2\pi]^N$, and that $\Pi$ is
equipped with the Lebesgue measure $(2\pi)^{-N} dt_1 \ldots dt_N$.
Renorming $X$, we assume that $\|f\|_{L_1} \leq \|f\|$ for any $f \in X$.
Let $C$ be a constant for which $C \|f\|_\infty \geq \|f\|$.
The dictionary ${\mathcal{D}}$ consists of functions
$f_\alpha(t) = \exp(i \langle \alpha, t\rangle)$ ($\alpha \in \R^N$).
We shall show that ${\mathcal{D}}$ has Property (P).
To this end, fix $n \in \N$, and let
$x = \sum_{k=1}^{n^2} (-1)^k f_{(k,0, \ldots, 0)}/n^2$
(note that $f_{(k,0, \ldots, 0)}(t_1, \ldots, t_N) = \exp(i k t_1)$).
Clearly, $\|x\| \leq C$. Consider a family
$\alpha^{(j)} = (\alpha^{(j)}_1, \ldots, \alpha^{(j)}_N) \in \R^N$
($1 \leq j \leq m \leq n-1$), and scalars $a_1, \ldots, a_m$.
Let $y = x + z$, where $z = \sum_{j=1}^m a_j f_{\alpha^{(j)}}$.
We shall show that $\|y\| \geq 1/n^2$.

Perturbing the $\alpha^{(j)}$'s slightly, we can assume that
all the quantities $\alpha^{(j)}_1$ are different, and non-integer.
To estimate $\|y\|$, recall that, for a multiindex $k = (k_1, \ldots, k_N) \in \Z^N$,
and a function $\phi$ defined on $\Pi$, we define the Fourier coefficient
$$
\hat{\phi}(k) = \langle \phi, f_k \rangle =
\frac{1}{(2\pi)^N} \int_0^{2\pi} \ldots \int_0^{2\pi}
\phi(t_1, \ldots, t_N) \exp(- i (k_1 t_1 + \ldots + k_N t_N)) \, d t_1 \ldots d t_N .
$$
We shall show that, for at least one value $k \in \{1, \ldots, n^2\}$, $(-1)^k \Re(\hat{z}(k,0,\ldots,0)) \geq 0$.
Once this is done, we conclude that
\begin{align*}
\|y\|
&
\geq
\|y\|_1 \geq |\hat{y}(k,0,\ldots,0)| =
|\hat{x}(k,0,\ldots,0) + \hat{z}(k,0,\ldots,0)|
\\
&
=
|\frac{(-1)^k}{n^2} + \hat{z}(k,0,\ldots,0)| \geq
\big|\Re\big(\frac{(-1)^k}{n^2} + \hat{z}(k,0,\ldots,0)\big)\big| \geq 1/n^2 ,
\end{align*}
which is what we need.

A straightforward calculation shows that, for
$\alpha = (\alpha_1, \ldots \alpha_N)$,
$$
\hat{f_\alpha}(k,0, \ldots,0) =
\frac{c_{\alpha_1} \ldots c_{\alpha_N}}{(\alpha_1 - k) \alpha_2 \ldots \alpha_N}
, \, \, \,
{\mathrm{where}} \, \,  c_\beta = \frac{\exp(i\beta) - 1}{2 \pi i} .
$$
Let $b_j = \Re\big(a_j c_{\alpha^{(j)}_1} \ldots c_{\alpha^{(j)}_N}\big)$.
Suppose, for the sake of contradiction,
$$
\Re\big(\hat{z}(k,0,\ldots,0)\big) = \sum_{j=1}^m b_j
\frac{1}{\alpha_1^{(j)} - k} \frac{1}{\alpha_2^{(j)}} \ldots \frac{1}{\alpha_N^{(j)}}
$$
has the same sign as $(-1)^{k+1}$
for every value of $k$. As $m < n$, there exists
$L \in \{1, \ldots, n(n-1)\}$ such that $[L, L+n-1] \cap \{\alpha^{(1)}_1, \ldots, \alpha^{(m)}_1\} = \emptyset$.
Indeed, $\{\alpha^{(1)}_1, \ldots, \alpha^{(m)}_1\}$ partitions
$[1,n^2]$ into no more than $m+1$ subintervals. If each of these
subintervals contains less than $n$ integer points, then
the total number of integer points on $[1,n^2]$ cannot exceed
$(n-1)(n+1)$, which is clearly false.

For $t \in [L, L+n-1]$ define
$$
\phi(t) = \sum_{j=1}^m b_j
\frac{1}{\alpha_1^{(j)} - t} \frac{1}{\alpha_2^{(j)}} \ldots \frac{1}{\alpha_N^{(j)}}
$$
By assumption, $\phi(k)$ is positive when $k$ is  an odd  
integer, and negative if $k$ is an even
integer. Therefore, for $s \in \{1, \ldots, n-1\}$
there exists $t_s \in (L+s-1,L+s)$ such that $\phi(t_s) = 0$.

Now consider the $m \times m$ matrix
$$
A = \Big[ \Big(
\frac{1}{\alpha_1^{(j)} - t_s} \frac{1}{\alpha_2^{(j)}} \ldots \frac{1}{\alpha_N^{(j)}}
\Big)_{j,s=1}^m \Big] ,
$$
and the vector $b = (b_1, \ldots, b_m)^t$. Then $A b = 0$, hence the matrix $A$ is singular.
However, by Cauchy's Lemma (see e.g. \cite[p.~195]{cheney}), the determinant
of the matrix with entries $((x_i - y_j)^{-1})$ equals
$\prod_{i<j}(x_i - x_j)(y_i - y_j)/\prod_{i,j}(x_i + y_j)$, hence
$A$ is non-singular.
\end{proof}

Theorem~\ref{ridge} can be connected to the problem of approximation
by elements of a frame (see e.g.~\cite{casazza_art} for an introduction
to the topic). By Corollary 3.10 of \cite{han_larson}, any normalized
tight frame ${\mathcal{F}}$ in a Hilbert space of the form $(U^n \eta)_{n \in \Z}$
($U$ is a unitary operator) is unitarily equivalent to the set ${\mathcal{D}}$
of the functions $t \mapsto \exp(2 \pi i t)|_E$, where $E$ is an essentially
unique measurable subset of $[0,2\pi]$. If $E$ contains an interval,
Theorem~\ref{ridge} shows that ${\mathcal{D}}$ (and therefore,
${\mathcal{F}}$) satisfies Shapiro's Theorem. We do not know
whether this remains true for general sets $E$.

In general, a frame ${\mathcal{D}}$ need not satisfy Shapiro's Theorem.
For instance, we can find a family of vectors $(u_i^{(j)})_{i,j \in \N}$,
dense in $S(\ell_2)$, such that $(u_i^{(j)})_{i \in \N}$ is an orthonormal basis
for every $j$. If $\sum_j |\alpha_j|^2 = 1$, then
${\mathcal{D}} = (\alpha_j u_i^{(j)})_{i,j \in \N}$ is a tight frame
(that is, $\sum_{e \in {\mathcal{D}}} |\langle f, e\rangle|^2 = \|f\|^2$
for any $f \in \ell_2$), yet clearly ${\mathcal{D}}$ fails Shapiro's Theorem.
Frames which are ``not too rich'', however, do satisfy Shapiro's Theorem.
For instance, suppose a frame has {\it finite excess} -- that is,
the removal of finitely many elements turns it into a basis (see
e.g.~\cite{liu_zheng} for some remarkable properties of frames with
finite excess). Theorem~\ref{decomposition} shows that such frames
satisfy Shapiro's Theorem.
Another class of interest is that of {\it Riesz frames} -- that is, of frames
$(f_i)_{i \in I}$ for which
there exist positive constants $A \leq B$ such that, for every $J \subset I$, and
every $f \in \spn[f_i : i \in J]$,
$A \|f\|^2 \leq \sum_{i \in J} |\langle f_i, f \rangle|^2 \leq B \|f\|^2$.

\begin{proposition}\label{riesz_frame}
If a dictionary ${\mathcal{D}}$ is a Riesz frame in $\ell_2$, then it satisfies
Shapiro's Theorem.
\end{proposition}

\begin{proof}
By \cite[Theorem 2.4]{casazza_subframe}, we can represent ${\mathcal{D}}$ as a union
of two disjoint subsets: an unconditional basis $(g_i)_{i \in \N}$,
and a family $(h_i)_{i \in \Gamma}$ ($\Gamma$ may be finite or infinite),
such that, for every $i \in \Gamma$, there exists a set $\Delta_i$ such that
$h_i \in \spn[g_j : j \in \Delta_i]$, and $K = \sup_{i \in \Gamma} |\Delta_i| < \infty$.

By \cite[Proposition 4.3]{casazza_art}, there exist $0 < C \leq D$
(depending only on $A$ and $B$) with the property that
$C^2 \sum_j |\alpha_j|^2 \leq \|\sum_j \alpha_j g_j\|^2 \leq D^2 \sum_j |\alpha_j|^2$
for any $(\alpha_j) \in \ell_2$. Consider $y = \sum_{j=1}^{2nK} g_j/(D \sqrt{2nK})$ and
$$
z = \sum_{i \in {\mathcal{A}}} \alpha_i h_i + \sum_{j \in {\mathcal{B}}} \beta_j g_j
$$
with $|{\mathcal{A}}| + |{\mathcal{B}}| \leq n$. Then $\|y\| \leq 1$.
We show that $\|y - z\| \geq C/(D \sqrt{2})$.
As $(g_j)$ is a basis, we can write $y - z = \sum_{j=1}^\infty \gamma_j g_j$.
Note that, for
$$
j \in {\mathcal{C}} = \{1, \ldots, 2nK\} \backslash \big( {\mathcal{B}} \cup
( \cup_{i \in {\mathcal{A}}} \Delta_i ) \big) ,
$$
$\gamma_j = 1/(D \sqrt{2nK})$. As $|{\mathcal{C}}| \geq nK$, we conclude that
$\|y-z\| \geq C/(D \sqrt{2})$. Since this inequality holds for any
$z \in \Sigma_n({\mathcal{D}})$, Corollary~\ref{coro} completes the proof.
\end{proof}

Certain approximation schemes related to MRA wavelets also satisfy Shapiro's
Theorem. In the exposition below, we follow the notation of \cite{Wo}.
Suppose $\phi$ is a scaling function in $L_2(\R)$. More precisely, suppose
$\|\phi\| = 1$. For $k, j \in \Z$, let
$\phi_{k,j}(x) = 2^{k/2} \phi(2^k x - j)$. Let
$V_k = \spn[\{ \phi_{k,j} : j \in \Z\}]$. We are assuming that
$V_k \subset V_{k+1}$ for any $k \in \Z$,
$\overline{\cup_k V_k} = L_2(\R)$, and $\cap_k V_k = \{0\}$.
Moreover, we assume that $\{\phi_{0,j}\}_{j \in \Z}$ is an orthonormal
basis for $V_0$. Now we consider the dictionary ${\mathcal{D}} = \{\phi_{k,j} : k, j \in \Z\}$
in $L_p(\R)$, for $1 < p < \infty$ and its associated n-term approximation scheme $A_n = \Sigma_n({\mathcal{D}})$.


\begin{theorem}\label{scaling}
In the above notation, suppose the scaling function $\phi$ has compact support.
Then:
\begin{itemize}
\item[$(i)$] $(L_2(\R),\{A_n\})$ satisfies Shapiro's Theorem.
\item[$(ii)$]
Suppose, furthermore, that $\phi \in L_\infty(\R)$.
Then  $(L_p(\R),\{A_n\})$ satisfies Shapiro's Theorem for $1 \leq p < \infty$.
\end{itemize}
\end{theorem}

Note first that the orthogonal projection from $L_2(\R)$
onto $V_k$ is given by
\begin{equation}
P_k f = \sum_{j \in \Z} \phi_{k,j} \int_\R f(t) \overline \phi_{k,j}(t) \, dt .
\label{orth_proj}
\end{equation}
If $\phi \in L_\infty(\R)$, then this family of projections is also
uniformly bounded on $L_p(\R)$, for any $p \in [1,\infty)$
(see \cite[Section 8.1]{Wo} for the proof of this fact, and
for further properties of these projections).

Define the map $D_k$ by setting $D_k f(x) = f(2^k x)$. Then
$V_k = D_k(V_0)$ for any $k$, and  $P_k = D_k P_0 D_{-k}$.

\begin{lemma}\label{almost_orth}
Suppose $\phi$ is a scaling function in $L_2(\R)$ with compact support.
Then, for any $n \in \N$ and $\varepsilon > 0$, there exists $N \in \N$
with the property that, for any $S \subset \{N, N+1, \ldots, \}\times \Z$ of
cardinality $n$ or less, and any $f = \sum_{s \in S} \alpha_s \phi_s$,
$\|P_0 f\| \leq \varepsilon \|f\|$.
\end{lemma}

\begin{proof} 
Let $T = \{t \in \R : \phi(t) \neq 0\}$, and $T_{k,j} = 2^{-k/2}(T+j)$.
By assumption, $T$ is (up to a set of measure zero) a subset of
a certain interval $I$, of length $|I|$.
Then $T_{k,j}$ belongs to an interval of length $2^{-k/2} |I|$.
Thus, there exists a constant $K$ such that
$$
|\{\ell \in \Z : \left|T_{k,j} \cap T_{0,\ell}\right| > 0\} |\leq K
$$
for any $k \geq 0$ and $j \in \Z$.

Consider $f = \sum_{i=1}^n \alpha_i \phi_{k_i, j_i}$,
with $k_i \geq 0$. Let $T_f = \cup_i T_{k_i,j_i}$. Then
$$
|\{\ell \in \Z : \left|T_f \cap T_{0,\ell}\right| > 0\}| \leq K n .
$$
Find $N \in \N$ such that
$|\langle g, \phi \rangle| \leq \varepsilon \|g\|/Kn$
whenever $g$ differs from $0$ on a set of measure at most $n 2^{-N} d$.
Then, for any $f = \sum_{i=1}^n \alpha_i \phi_{k_i, j_i}$,
with $k_i \geq N$, $|\langle f, \phi_{0,j} \rangle| \leq \varepsilon \|f\|/(Kn)$.
Moreover, $\langle f, \phi_{0,j} \rangle \neq 0$ for at most $Kn$ different
values of $j$. To complete the proof, recall that, by \eqref{orth_proj},
$$
P_0 f = \sum_{j \in \Z} \langle f, \phi_{0,j} \rangle \phi_{0,j} .
$$
\end{proof} 

A variant of the previous lemma (with identical proof) also holds for
$1 < p < \infty$.

\begin{lemma}\label{almost_orth_p}
Suppose $\phi$ is a scaling function in $L_\infty(\R)$ with compact support,
and $1 \leq p < \infty$.
Then, for any $n \in \N$ and $\varepsilon > 0$, there exists $N \in \N$
with the property that, for any $S \subset \{N, N+1, \ldots, \}\times \Z$ of
cardinality $n$ or less, and any $f = \sum_{s \in S} \alpha_s \phi_s$,
$\|P_0 f\|_p \leq \varepsilon \|f\|_p$.
\end{lemma}

\begin{proof}[Proof of Theorem~\ref{scaling}]
We prove part $(i)$ only, since part $(ii)$ is handled in the
same manner, with minimal changes.
By Theorem~\ref{condicion}, it suffices to show that this
approximation scheme has Property (P). To this end, fix $n \in \N$.
Let $c = 1/(8\sqrt{n+1})$.
By Lemma~\ref{almost_orth}, there exists $N \in \N$
 such that $\|P_0 f\| \leq c \|f\|$
for any $f = \sum_{i=1}^n \alpha_i \phi_{k_i, j_i}$
whenever $k_i \geq N-1$ for each $i$. It is easy to see that,
for any $m \in \Z$, we have $\|P_m f\| \leq c \|f\|$
for any $f = \sum_{i=1}^n \alpha_i \phi_{k_i, j_i}$
whenever $k_i \geq m+N-1$ for each $i$.

Find norm $1$ vectors
$x_s \in V_{sN+1} \ominus V_{sN}$ ($0 \leq s \leq n$).
Let $x = (x_0 + \ldots + x_n)/\sqrt{n+1}$.
We show that $E(x, A_n) \geq c$.
Indeed, consider $f = \sum_{i=1}^n \alpha_i \phi_{k_i, j_i} \in A_n$,
and suppose, for the sake of contradiction, that $\|x - f\| < c$.
By Pigeon-Hole Principle, there exists $s \in \{0, \ldots, n\}$
with the property that no $k_i$ belongs to $\{sN, \ldots, (s+1)N-1\}$.
Let $f_- = \sum_{k_i \leq sN} \alpha_i \phi_{k_i, j_i}$,
and $f_+ = \sum_{k_i \geq (s+1)N} \alpha_i \phi_{k_i, j_i}$.
Note that, by our choice of $N$,
$\|P_m f_+\| \leq c\|f_+\|$ whenever $m \leq sN+1$.

In this notation,
$$
c > \|x - f\| \geq \|(I - P_{sN})(x - f)\| =
\|\frac{x_s + \ldots + x_n}{\sqrt{n+1}} - (I - P_{sN}) f_+\| .
$$
By the triangle inequality, $\|(I - P_{sN}) f_+\| < 1 + c$.
Therefore, $\|f_+\| < 2$. Indeed, otherwise we would have
$$
1+c > \|(I - P_{sN}) f_+\| \geq \|f_+\| - \|P_{sN} f_+\| \geq
(1-c) \|f_+\| \geq 2 (1 - c) ,
$$
which contradicts the fact that $c < 1/8$.

Similarly,
$$
c > \|x - f\| \geq \|(I - P_{sN+1})(x - f)\| =
\|\frac{x_{s+1} + \ldots + x_n}{\sqrt{n+1}} - (I - P_{sN+1}) f_+\| .
$$
Thus, by the triangle inequality,
\begin{align*}
2c
&
>
\| \Big( \frac{x_s + \ldots + x_n}{\sqrt{n+1}} - (I - P_{sN}) f_+ \Big) -
\Big( \frac{x_{s+1} + \ldots + x_n}{\sqrt{n+1}} - (I - P_{sN+1}) f_+ \Big) \|
\\
&
=
\|\frac{x_{s+1}}{\sqrt{n+1}} + P_{sN} f_+ - P_{sN+1} f_+ \| \geq
\|\frac{x_{s+1}}{\sqrt{n+1}}\| - \|P_{sN} f_+\| - \|P_{sN+1} f_+\| .
\end{align*}
We know that $\|P_m f_+\| \leq c \|f_+\| \leq 2c$ for any $m \leq sN+1$.
Recall that $\|x_{s+1}/\sqrt{n+1}\| = 8c = 1/\sqrt{n+1}$.
The previous centered inequality then implies $2c > 8c - 2c - 2c = 4c$,
a contradiction.
\end{proof} 

Next we deal with the dictionaries in $L_p(\R)$
or $C_0(\R)$ arising from translates of
a single function. More precisely, for $\phi \in L_p(\R)$,
consider the set ${\mathcal{D}} = \{\phi_c : c \in \R\}$,
with $\phi_c(t) = \phi(t-c)$.
It is a well known result by Wiener (see \cite[pp.97-103]{Wi}, or \cite[Chapter 8]{edwards}) that $\spn[{\mathcal{D}}]$ is dense in $L_1(\R)$ if and only if the Fourier transform of $\phi$ doesn't vanish on $\R$,
and $\spn[{\mathcal{D}}]$ is dense in $L_2(\R)$ if and only if
the Fourier transform of $\phi$ vanishes only on a measure $0$ subset of $\R$.
This condition is satisfied, for instance, if $\phi$ is a Gaussian
function $\phi(t) = e^{-at^2/2}$, for some $a>0$.

\begin{theorem}\label{translates}
Suppose $X$ is either $L_p(\R)$ ($0 < p < \infty$) or $C_0(\R)$, and
$\phi$ is a function in $X$. Denote by ${\mathcal{D}}$ the set of translates
$\{\phi_c : c \in \R\}$. Then the approximation scheme
$(X,\{\Sigma_n({\mathcal{D}})\})$ satisfies Shapiro's Theorem in each of the following two cases:
\begin{enumerate}
\item
$\phi$ has compact support, and the linear span of its translates is dense in $X$.
\item
$\phi$ is a Gaussian function.
\end{enumerate}
\end{theorem}

\begin{proof}
(1) We consider the case of $X = L_p(\R)$. The space $C_0(\R)$ can
be tackled in a similar fashion.
By Corollary~\ref{coro}, it suffices to show that, for any $n \in \N$ and
$\varepsilon > 0$, there exists $f \in L_p(\R)$ such that
$E(f, \Sigma_n({\mathcal{D}})) > 1 - \varepsilon$.
To this end, pick $m \in \N$ such that  $n/m < \varepsilon$.
Find a finite interval $I$ such that $\phi$ vanishes outside of $I$.
Set $f = m^{-1/p} \sum_{i=1}^m \chi_{[ai, ai+1]}$, where
$a = |I| + 2$, and $m > n$. Consider
$g = \sum_{j=1}^n \alpha_j \phi_{c_j} \in \Sigma_n({\mathcal{D}})$.
Then $g$ vanishes outside $S = \cup_{j=1}^n (I+c_j)$. By definition of $a$,
$[ai,ai+1] \cap S = \emptyset$ for at least $m-n$ values of $i$. Therefore,
$\|f-g\| \geq ((m-n)/m)^{1/p}$, which is what we need.

(2)
By \cite[Theorem 2]{zalik}, $\spn[{\mathcal{D}}]$
is dense in $C_0(\R)$, hence also in $L_p(\R)$ ($0 \leq p < \infty$).
A Bernstein-type inequality from \cite{erd} shows that,
for any $f \in \Sigma_n({\mathcal{D}})$, we have
$\|f^\prime\|_p \leq c n^{1/2} \|f\|_p$, for any $p \in (0,\infty]$.
An application of Theorem~\ref{central} completes the proof.
\end{proof} 

Finally we consider approximation schemes in tensor products and operator ideals.
Suppose $X$ and $Y$ are Banach spaces. A {\it cross-norm} $\alpha$ on the algebraic
tensor product $X \otimes Y$ is a norm satisfying $\|x \otimes y\| = \|x\| \|y\|$
for any $x \in X$ and $y \in Y$. The completion of $X \otimes Y$ with respect to
this norm is denoted by $X \otimes_\alpha Y$ (this is a Banach space).
The reader is referred to e.g.~\cite{DJT, PieID, T-J} for information about tensor norms.

\begin{proposition}\label{tensor}
Suppose $X$ and $Y$ are infinite dimensional Banach spaces, and $\alpha$ is
a cross-norm. Denote by $A_n$ the set of sums $\sum_{j=1}^n x_j \otimes y_j$ in
$Z = X \otimes_\alpha Y$. Then $(A_n)$ is an approximation
scheme in $Z$, satisfying Shapiro's Theorem.
\end{proposition}

\begin{proof}
Obviously, $A_n=\Sigma_n(\mathcal{D})$, where $\mathcal{D}=X\otimes Y$ is,
by definition of $Z$, a dense subset of $Z$.
To tackle Shapiro's Theorem, we show that $(A_n)$ has
Property $(P)$. Fix $n$.
Find unit vectors $(x_i)_{i=1}^n$ and $(y_i)_{i=1}^n$ in $X$ and $Y$,
respectively, forming Auerbach bases in their respective
linear spans. That is, for any scalars $\gamma_1, \ldots, \gamma_n$,
$$
\max |\gamma_k| \leq \min\{ \|\sum \gamma_k x_k\| , \|\sum \gamma_k y_k\|\}
\leq \sum |\gamma_k| .
$$
Let $z = \sum_{k=1}^n x_k \otimes y_k/n$. Then $\|z\| \leq 1$.
We show that $E(z, A_{n-1}) \geq 1/n^2$.

Suppose, for the sake of contradiction, that $\|z-c\| < 1/n^2$ for some
$c = \sum_{j=1}^{n-1} a_j \otimes b_j \in A_{n-1}$.
By Hahn-Banach Extension Theorem, there exist norm one linear functionals
$(f_k)$ and $(g_k)$ in $X^*$ and $Y^*$, respectively, which are biorthogonal to
$(x_k)$ and $(y_k)$. For $1 \leq p, q \leq n$,
$|\langle f_p \otimes g_q, z \rangle - \langle f_p \otimes g_q, c \rangle | < 1/n^2$.
However, $\langle f_p \otimes g_q, z \rangle = \delta_{pq}/n$,
hence $(\langle f_p \otimes g_q, z \rangle)_{p,q=1}^n = I/n$,
where $I$ is the $n \times n$ identity matrix.
On the other hand, the matrix
$d = (\langle f_p \otimes g_q, c\rangle)_{p,q=1}^n)$ has
rank less than $n$. Indeed, for each $j$, the rank of
$(\langle f_p, a_j \rangle \langle g_q, b_j \rangle)_{p,q}$ doesn't exceed $1$.
As $d = \sum_{j=1}^{n-1} (\langle f_p, a_j \rangle \langle g_q, b_j \rangle)_{p,q}$,
we conclude that $\rank d < n$.

Now equip the space of $n \times n$ matrices with the Hilbert-Schmidt
norm. It is well known that, for any matrix $A$ of rank less than $n$,
$\|I - A\|_{HS} \geq 1$. On the other hand,
$\|I/n - d\|_{HS}^2 = \sum_{p,q=1}^n |\delta_{pq} - d_{pq}|^2 < n^2 \cdot 1/n^4 = 1/n^2$,
a contradiction.
\end{proof}

Now suppose ${\mathcal{A}}$ is a quasi-Banach operator ideal, equipped with the norm
$\| \cdot \|_{\mathcal{A}}$ (see \cite{DJT, PieID, T-J}
for the definition and basic properties of operator ideals).
Define the {\it ${\mathcal{A}}$-approximation numbers} by setting
$$
a^{(\mathcal{A})}_n(T) = \inf_{u \in B(X,Y), \rank u < n} \|T-u\|_{\mathcal{A}} .
$$
Denote by $A^{(\mathcal{A})}(X,Y)$ the set of $\mathcal{A}$-approximable operators
-- that is, the operators T for which $\lim_n a^{(\mathcal{A})}_n(T) = 0$ --
equipped with the norm $\| \cdot \|_{\mathcal{A}}$. One can easily see this
is a quasi-Banach space (a Banach space if ${\mathcal{A}}$ is a Banach ideal).
Note that, if $\mathcal{A}$ is the ideal of bounded operators (or compact
operators), with its canonical norm, we obtain the usual definitions of
approximation numbers, and approximable operators, respectively.

Let $\mathcal{D}$ be the dictionary of rank $1$ vectors in
$A^{(\mathcal{A})}(X,Y)$, where $X$ and $Y$ are infinite dimensional
Banach spaces. For any $T \in B(X,Y)$, we have
$E(T, \Sigma_i(\mathcal{D})) = a^{(\mathcal{A})}_{i+1}(T)$.

\begin{corollary}\label{ideal}
In the above notation, the approximation scheme $(\Sigma_i(\mathcal{D}))$ satisfies
Shapiro's Theorem.
\end{corollary}

\begin{proof}
It is well known that $A^{(\mathcal{A})}(X,Y)$ can be identified with $X^* \otimes_\alpha Y$,
for the appropriate cross-norm $\alpha$. An application of Proposition~\ref{tensor}
completes the proof.
\end{proof} 

For certain ideals $\mathcal{A}$, this theorem can be strengthened:
it is possible to construct $T \in B(X,Y)$ for which the sequence
$(a_n^{(\mathcal{A})}(T))$ ``behaves like'' a prescribed sequence $(\alpha_n)$.
This result appears in the forthcoming paper \cite{Oi} of the
second author.

\section{Controlling the rate of approximation}\label{fast_decay}

In the previous sections of this paper, we proved that, for a number of
approximation schemes $(A_n)$, we can find and element $x$ in the ambient space,
for which the sequence $(E(x,A_n))$ decreases arbitrarily slowly.
In some situations, we can go further and guarantee a prescribed behavior of
$(E(x,A_n))$.

Recall that an approximation scheme $(X_n)$ in a Banach space $X$ is called
{\it linear} if the sets $X_n$ are linear subspaces of $X$.
By a classical result of Bernstein (see Section~\ref{intro}), if all the
$X_n$'s are finite dimensional and $\{\varepsilon_n\} \searrow 0$, then
there exists $x \in X$ such that $E(x,X_n) = \varepsilon_n$ for every $n \geq 0$.
Without the finite dimensionality assumption, things are different.
It was shown in \cite{nikolskii} (see also \cite[Section I.6.3]{singerlibro})
that a Banach space $X$ is reflexive if and only if for any finite sequence
of closed subspaces $\{0\} = X_0 \subsetneq X_1 \subsetneq X_2 \ldots \subsetneq X_n
\subsetneq X_{n+1} \subset X$, and for any $\varepsilon_0 \geq \varepsilon_1 \geq
\ldots \geq \varepsilon_n$, there exists $x \in X_{n+1}$ such that $E(x, X_k) = \varepsilon_k$ for any $0 \leq k \leq n$.
An inspection of the proof shows the following:

\begin{proposition}\label{nikolskii_modified}
Suppose $X$ is a Banach space, $\{0\} = X_0 \subsetneq X_1 \subsetneq X_2 \ldots
\subsetneq X_n \subsetneq X_{n+1} \subset X$ is a sequence of its
closed subspaces, and $\varepsilon_0 > \varepsilon_1 > \ldots > \varepsilon_n$.
Then there exists $x \in X_{n+1}$ such that
$E(x, X_k) = \varepsilon_k$ for any $0 \leq k \leq n$.
\end{proposition}

If the chain of subspaces $(X_n)$ is infinite, we obtain a somewhat weaker
result.

\begin{theorem}\label{control_decay}
Let $\{0\} = X_0 \subsetneq X_1 \subsetneq X_2 \ldots$ be a sequence
of closed subspaces of a Banach space $X$.
Then for every $\{\varepsilon_n\}\searrow 0$ and every $\{\delta_n\}\searrow 0$
there are $x\in X$, $C>0$ and $\{n(m)\}_{m=0}^{\infty}$  sequence of natural numbers verifying $n(m)\geq m$ for all $m$, such that
\[
\delta_m\varepsilon_{n(m)}\leq E(x,X_{n(m)})\leq C\varepsilon_{n(m)},\ \ m=0,1,2,\cdots.
\]
Moreover, there exists an strictly increasing sequence of natural numbers $\{h(m)\}_{m=0}^{\infty}$ such that
\[
\delta_{h(m)}\varepsilon_{h(m)}\leq E(x,X_{h(m)})\leq C\varepsilon_{h(m)},\ \ m=0,1,2,\cdots.
\]
\end{theorem}

\begin{proof}
Assume, without loss of generality, that $\varepsilon_0 = 1$.
Define $A_0 = A_0(\{\varepsilon_n\})$ as the set
$\{x\in X:\{\frac{E(x,X_n)}{\varepsilon_n}\}_{n=0}^{\infty}\in c_0\}$,
equipped with the norm $\|x\|_{A_0}=\sup_{n\in\mathbb{N}}\frac{E(x,X_n)}{\varepsilon_n}$
(see \cite[Prop. 3.8, Theorems 3.12 and 3.17]{almiraluther2}).
It is easy to see that $A_0$ is a Banach space, and the natural embedding
of $A_0$ into $X$ is contractive. We claim that, for any $x \in X$ and $m \geq 0$,
\begin{equation}
E(x,X_m)_{A_0(\{\varepsilon\})} = \sup_{n \geq m} \frac{E(x,X_n)}{\varepsilon_n} .
\label{approximate}
\end{equation}
Indeed,
\begin{equation}
E(x,X_m)_{A_0(\{\varepsilon\})} = \inf_{y\in X_m} \|x-y\|_{A_0} =
\inf_{y\in X_m} \sup_n \frac{E(x-y,X_n)}{\varepsilon_n} .
\label{def_of_E}
\end{equation}
For $n < m$, we trivially have $E(x-y,X_n) \leq \|x-y\|$.
For $m \geq n$, $E(x-y,X_n) = E(x,X_n)$. Taking the infimum over $y \in X_m$
in \eqref{def_of_E}, and recalling that $\varepsilon_0 \geq \varepsilon_1 \geq \ldots$,
we obtain \eqref{approximate}.

Note that $\lim_m \sup_{n \geq m} \varepsilon_n^{-1} E(x,X_n) = 0$,
hence $(X_n)$ is an approximation scheme in $A_0$. Moreover, this scheme
is non-trivial: for each $n$, the inclusion of $\overline{X_n}^{A_0}$ into $A_0$ is strict.
Thus, $(A_0,\{X_n\})$ satisfies Shapiro's Theorem.
By Corollary~\ref{coro}, for every $\{\delta_n\}\searrow 0$ there exists
$x\in A_0(\varepsilon_n)$ such that
\[
\sup_{n\geq m}\frac{E(x,X_n)}{\varepsilon_n}=E(x,X_m)_{A_0(\varepsilon_n)}\geq 2\delta_m \ \ (m=0,1,2,\cdots).
\]
In other words, for every $m\in\mathbb{N}$ there exists $n(m)\geq m$ such that
$E(x,X_{n(m)})\geq \delta_m\varepsilon_{n(m)}$. Taking $C=\|x\|_{A_0}$,
we establish the first claim of this theorem. To prove the second claim,
it is enough to take a strictly increasing subsequence $h(m)$ of $n(m)$, and
to recall that $\{\delta_m\}$ is decreasing.
\end{proof} 

Recall the {\it density sequence} $\dens_i = E(S(X),A_i)$, defined in
Section~\ref{fail_shapiro}. There, it was observed that $\{\dens_i\}_{i=0}^{\infty}$
is non-increasing, and $(X,\{A_n\})$ satisfies Shapiro's Theorem if and only if
$\dens_i = 1$ for every $i$. The following result is a ``mirror image'' of Brudnyi's theorem.

\begin{theorem}\label{fast}
Suppose that $\{\varepsilon_i\}$ is a sequence of positive numbers converging to
$0$, $(X,\{A_n\})$ is an approximation scheme in a Banach space $X$, and $\dens_i > 0$
for $i=0,1,\cdots$. Then there exists $x \in X \backslash (\cup_i A_i)$ such that
$0<E(x,A_i) \leq \varepsilon_i$ for each $i$.
\end{theorem}

\begin{lemma}\label{universal}
Suppose $(X,\{A_n\})$ is an approximation scheme in a Banach space $X$.
Suppose, furthermore, that $i \in \N$ satisfies $\dens_{K(i)} > 0$. Then for any
$c \in (0,1)$ there exist $j > i$ and $y_0 \in A_j$, such that $\|y_0\| = 1$, and
$E(x + \alpha y_0, A_i) > c |\alpha| \dens_{K(i)}$ for any $x \in A_i$,
and any scalar $\alpha$.
\end{lemma}

\begin{proof}
As $S(X) \bigcap (\bigcup_j A_j)$ is dense in $S(X)$,
we can find $j \in \N$ and $y_0 \in A_j \cap S(X)$ in such a way that
$E(y_0,A_{K(i)}) > c\dens_{K(i)}$.
Then, for any $x, z \in A_i$,
$$
\|(x + \alpha y_0) - z\| = \|\alpha y_0 - (z-x)\| \geq
|\alpha| E(y_0,A_{K(i)}) > c |\alpha| \dens_{K(i)},
$$
which is what we wanted to prove.
\end{proof}

\begin{proof}[Proof of Theorem~\ref{fast}]
We are going to find a ``rapidly increasing'' sequence
$0 = i_0 < 1 = i_1 < i_2 < i_3 < \ldots$, a ``rapidly
decreasing'' sequence $\delta_1 > \delta_2 > \ldots > 0$, and a sequence of elements $x_j \in
A_{i_j}$, in such a way that the following holds for every $j$:
\begin{equation}
\delta_j \leq \min\{ \varepsilon_{i_j}/2 , \delta_{j-1}
\dens_{K(i_{j-2})}/4 \} , \, \,
\|x_j - x_{j-1}\| \leq \delta_j , \, \,
E(x_j, A_{i_{j-1}}) > 4 \delta_j \dens_{K(i_{j-1})}/5 .
\label{conditions_delta}
\end{equation}
As $\delta_j \leq \delta_{j-1}/4$, $\{x_j\}$ is a  Cauchy
sequence in $X$.
Let $x = \lim_j x_j$. We claim that  $x \notin \cup_i A_i$ and
$E(x, A_\ell) < \varepsilon_\ell$ for each $\ell$.
Indeed, for $i_{j-1} \leq \ell < i_j$,
$$
E(x,A_\ell) \leq E(x,A_{i_{j-1}}) \leq \sum_{k \geq j} \delta_k \leq
\delta_j \sum_{s=0}^\infty 4^{-s} < 2 \delta_j \leq \varepsilon_{i_j} \leq
\varepsilon_\ell .
$$
On the other hand,
$$
E(x,A_{i_j}) \geq \frac{4 \dens_{K(i_{j-1})}\delta_j}{5} - \sum_{k>j}
\delta_k \geq
\frac{4 \dens_{K(i_{j-1})}\delta_j}{5} - \dens_{K(i_{j-1})} \delta_j
\sum_{s=1}^\infty 4^{-s} >
\frac{\dens_{K(i_{j-1})}\delta_j}{3} > 0 .
$$

Thus, it suffices to show the existence of the sequences $\{i_j\}$, $\{x_j\}$,
and $\{\delta_j\}$ with desired properties.
Set $x_0 = 0$. Let $\delta_1 = \varepsilon_1/2$,
and pick an arbitrary $x_1 \in A_1$ with $\|x_1\| = \delta_1$.
Now suppose $x_j\in A_{i_j}$, $\delta_j>0$, and $n_j\in \N$ have been defined for $j < k$,
in such a way that \eqref{conditions_delta} are satisfied.
By Lemma~\ref{universal}, we can find $s$ such that
there exists $y \in A_s$ with $\|y\| = 1$, for which
$E(x_{k-1} + \delta y, A_{i_{k-1}}) > 4 \delta \dens_{K(i_{k-1})}/5$
hold for any $\delta > 0$. Set $i_k = K(s)$, and
$\delta_k = \min\{\varepsilon_{i_k} , \delta_{k-1} \dens_{K(i_{k-2})}/4\}$.
Then $x_k = x_{k-1} + \delta_k y \in A_{i_k}$,
$\|x_k - x_{j-k}\| = \delta_k$, and
$E(x_k, A_{i_{k-1}}) > 4 \dens_{K(i_{k-1})} \delta_k/5$.
\end{proof}


\bigskip

\footnotesize{J. M. Almira

Departamento de Matem\'{a}ticas. Universidad de Ja\'{e}n.

E.P.S. Linares,  C/Alfonso X el Sabio, 28

23700 Linares (Ja\'{e}n) Spain

Email: jmalmira@ujaen.es

Phone: (34)+ 953648503

Fax: (34)+ 953648575}

\bigskip

\footnotesize{T. Oikhberg

Department of Mathematics, The University of California at Irvine, Irvine CA 92697, {\it and}

Department of Mathematics, University of Illinois at Urbana-Champaign, Urbana, IL 61801

Email: toikhber@math.uci.edu

Phone: (1)+ 949-824-1267

Fax: (34)+ 949-824-7993}

\end{document}